\documentclass{amsart}
\usepackage{amsmath}
\usepackage{amssymb}
\usepackage{amsthm}
\usepackage{tikz}
\usepackage{hyperref}

\newtheorem{theorem}{Theorem}[section]
\newtheorem{remark}[theorem]{Remark}

\newtheorem{lemma}[theorem]{Lemma}

\newtheorem{proposition}[theorem]{Proposition}
\newtheorem{corollary}[theorem]{Corollary}

\DeclareMathOperator{\sgn}{sgn}
\begin{document}
\title{Twisted  characteristic $p$ zeta  functions}

\author{Bruno Angl\`es  \and Tuan Ngo Dac  \and Floric Tavares Ribeiro}

\address{
Universit\'e de Caen Normandie,
Laboratoire de Math\'ematiques Nicolas Oresme,
CNRS UMR 6139,
Campus II, Boulevard Mar\'echal Juin,
B.P. 5186,
14032 Caen Cedex, France.
}
\email{bruno.angles@unicaen.fr,tuan.ngodac@unicaen.fr,floric.tavares-ribeiro@unicaen.fr}

\begin{abstract}
We propose a ``twisted" variation of zeta functions  introduced by David Goss in 1979.
\end{abstract}

\date{\today}

\maketitle

\tableofcontents

%%%%%%%%%%%%%%%%%%%%%%% INTRODUCTION%%%%%%%%%%%%%%%%%%%%%%%%%%%%%%%%%%%%%%%%%%%%%%%%%%

\section{Introduction}\label{Introduction}${}$\par

Let $\mathbb F_q$ be a finite field having $q$ elements, and let $\theta$ be an indeterminate over $\mathbb F_q;$ the Carlitz-Goss zeta function is defined as follows (\cite{GOS}, chapter 8):
$$\zeta_A((x; y)):=\sum_{d\geq 0}(\sum_{a\in A_{+,d}}\frac{1}{(\frac {a}{\theta^{d}})^y})x^{-d}\in \mathbb C_\infty^\times,$$
where $(x;y)\in \mathbb C_\infty^\times \times \mathbb Z_p,$ $A:=\mathbb F_q[\theta],$ $A_{+,d}:=\{ a\in A, a\, {\rm monic}\, , \deg_\theta a=d\},$ and $\mathbb C_\infty$ is the completion of an algebraic closure of $\mathbb F_q((\frac{1}{\theta})).$  The facts that $\zeta_A(.)$ converges on $\mathbb S_\infty:= \mathbb C_\infty^\times \times \mathbb Z_p$ and is ``essentially algebraic" (i.e. for $y\in \mathbb N,$ $\zeta_A((\theta^{y}x; -y))$ is a polynomial in $x^{-1}$ with coefficients in $A$) can be proved by using the following  vanishing result: for $n\in \mathbb  N, $ for $d(q-1)> \ell_q(n),$ $\sum_{a\in A_{+,d}} a^n =0,$ where $\ell_q(n)$ is the sum of digits of $n$ in base $q$ (this is a consequence of the combinatorial result: \cite{GOS}, Lemma 8.8.1).  The Carlitz-Goss zeta function is an example of a new kind of $L$-series introduced by D. Goss in \cite{GOS1} (see also \cite{GOS2}). The ``special values" of these type of $L$-functions are at the heart of function field arithmetic and we refer the interested reader to (this list is clearly  not exhaustive): \cite{AND&THA}, \cite{BOC}, \cite{LAF},  \cite{TAE2}.\par
Let $s,n\in \mathbb N,$ $n\geq 1,$ let $x\in \mathbb C_\infty^\times$ and $y_1, \ldots, y_n \in \mathbb Z_p;$ let's consider the following ``twisted" version of the Carlitz-Goss zeta function:
$$\zeta_A(t_1,\ldots, t_s; \theta_1, \ldots, \theta_n; (x; y_1, \ldots, y_n)):=
\sum_{d\geq 0}(\sum_{a\in A_{+,d}}\frac{a(t_1)\cdots a(t_s)}{(\frac {a(\theta_1)}{\theta_1^{d}})^{y_1}\cdots(\frac {a(\theta_n)}{\theta_n^{d}})^{y_n} })x^{-d},$$
where $t_1, \ldots, t_s\in \mathbb C_\infty,$ $\theta_1, \ldots, \theta_n \in \mathbb C_\infty^\times$ and are such that $\frac{1}{\theta_i}, i=1, \ldots, n,$ is in the maximal ideal of the valuation ring of $\mathbb C_\infty.$ Using the fact that $\forall t_1, \ldots, t_s \in \mathbb C_\infty,$  $\forall d>\frac{s}{q-1},$ $\sum_{a\in A_{+,d}} a(t_1)\cdots a(t_s)=0$ (this is again a consequence of \cite{GOS}, Lemma 8.8.1),   one can prove that $\zeta_A(t_1,\ldots, t_s; \theta_1, \ldots, \theta_n;.)$ converges on $\mathbb C_\infty^\times \times \mathbb Z_p^n$ and that this function is essentially algebraic, i.e. for $y_1, \ldots, y_n \in \mathbb N$:  $$\zeta_A(t_1,\ldots, t_s; \theta_1, \ldots, \theta_n; (\prod_{i=1}^n \theta_i^{y_i}x; -y_1, \ldots, -y_n))\in \mathbb F_q[t_1, \ldots, t_s,\theta_1, \ldots, \theta_n, x^{-1}].$$
Observe that if $s=0,n=1$ and $\theta_1=\theta,$ we recover the Carlitz-Goss zeta function, and if $s\geq 0,$ $n=1,$ $\theta_1=\theta,$ we recover the $L$-series introduced by F. Pellarin in \cite{PEL}.\par
Our aim in this article is to extend the above construction to the case where $K/\mathbb F_q$ is a global function field, $\infty$ is a place of $K,$ and $A$ is the ring of elements of $K$ which are regular outside $\infty.$ Although the twisted zeta functions proposed in this paper are in the spirit of the twisted Carlitz-Goss zeta functions defined above, the proof of the convergence of these functions, and their $v$-adic interpolation at finite places $v$ of $K,$  are  more subtle and  based on a technical key  Lemma which generalizes the vanishing result mentioned above (Lemma \ref{LemmaS3.2}).  The zeta functions introduced by D. Goss as well as their deformations over affinoid algebras are examples of such twisted zeta functions (see example \ref{GOSS}).\par
The main ingredient to our construction is what we called \textit{admissible maps} (section \ref{Notation}). Let $K_\infty$ be the $\infty$-adic completion of $K,$ let $\mathbb F_\infty$ be the residue field of $K_\infty,$ and let $\sgn:K_\infty^\times \rightarrow \mathbb F_\infty^\times$ be a ``sign function", i.e. $\sgn$ is a group homomorphism such that $\sgn\mid_{\mathbb F_\infty^\times}={\rm Id}_{\mathbb F_\infty^\times}.$ Let $\pi$ be a prime of $K_\infty$ such that $\pi\in {\rm Ker} \, \sgn.$ Let $\overline{K}_\infty$ be a fixed algebraic closure of $K_\infty,$ let $\overline{\mathbb F}_q$ be the algebraic closure of $\mathbb F_q$ in $\overline{K}_\infty,$ then $\sgn$ can be naturally extended to a morphism $\sgn:\overline{K}_\infty\rightarrow \overline{\mathbb F}_q^\times$ (see section \ref{Notation}) according to our  choice of $\pi.$ Let $v_\infty:\overline{K}_\infty\rightarrow \mathbb Q\cup\{+\infty\}$ be the valuation on $\overline{K}_\infty$ such that $v_\infty(\pi)=1.$ For $x\in \overline{K}_\infty^\times,$ let's set:
$$\langle x \rangle =\frac{x}{\sgn (x)\pi^{v_\infty(x)}}.$$
Let $\mathcal I(A)$ be the group of non-zero fractional ideals of $A.$ By definition, an admissible map $\eta: \mathcal I(A) \rightarrow \overline{K}_\infty$ is a map such that there exist $n(\eta)\in \mathbb Z_p,$ $\gamma(\eta)\in \overline{K}_\infty^\times,$ and an open subgroup of finite index $M(\eta)\subset K_\infty^\times,$ with the following property:
$$\forall I\in \mathcal I(A), \forall a\in K^\times \cap M(\eta),\quad \eta(Ia)=\eta (I) \langle a \rangle ^{n(\eta)} \gamma_\eta^{v_\infty(a)}.$$
The twisted zeta functions considered in this work are constructed with the help of such admissible maps (section \ref{SVTZF}). Let's give a basic example. Let $\eta: \mathcal I(A)\rightarrow \overline{K}_\infty^\times$ be an admissible map such that $n(\eta)=1 $ and $\eta$ is algebraic (i.e. $\eta(\mathcal I(A))\subset \overline{K}^\times,$ where $\overline{K}$ is the algebraic closure of $K$ in $\overline{K}_\infty$). Let's set:
$$K(\eta)=K(\mathbb F_\infty,\gamma_\eta \pi^{-1},\eta_i(I), I\in \mathcal I(A)),$$
$$K_\infty(\langle \eta \rangle )=K_\infty( \langle \gamma_\eta \rangle , \langle \eta_i(I) \rangle , I\in \mathcal I(A)).$$
Observe that $K(\eta)/K$ is a finite extension, and $K_\infty(\langle \eta \rangle )/K_\infty$ is also a finite extension. Let $s,n\in \mathbb N,$ $n\geq 1.$ Let $\mathbb C_\infty$ be the completion of $\overline{K}_\infty.$ Let $\rho_1, \ldots \rho_s: K(\eta)\hookrightarrow \mathbb C_\infty$ be $s$ $\mathbb F_p$-algebra homomorphisms, and let $\sigma_1, \ldots , \sigma_n: K_\infty(\langle \eta \rangle )\hookrightarrow \mathbb C_\infty$ be $n$ continuous $\mathbb F_p$-algebra homomorphisms. For $u=(x; (y_1, \ldots, y_n))\in \mathbb C_\infty^\times\times \mathbb Z_p^n,$ let's set:
$$\zeta_{\eta,A}(\underline{\rho}; \underline{\sigma}; u)= \sum_{d\geq 0}(\sum_{\substack{I\in \mathcal I(A), I\subset A \\ \deg I=d}} \frac{\prod_{i=1}^s \rho_i(\eta(I))}{\prod_{j=1}^n \sigma_j(\langle \eta(I) \rangle ^{-{y_j}})})x^{-d}.$$
Then, as a consequence of Theorem \ref{TheoremS4.1}, $\zeta_{\eta,A}(\underline{\rho}; \underline{\sigma}; .)$ converges on $\mathbb C_\infty^\times\times \mathbb Z_p^n.$ One can easily notice that the above function, and its twists by some characters, generalize the twisted Carlitz-Goss zeta function. We refer the reader to section \ref{Examples} for  more detailed  examples.\par

For the sake of completeness, in the last section of this article, we also introduce multiple twisted zeta functions, as well as their $v$-adic interpolation at  finite places $v$ of K,  in the spirit of Thakur's multiple zeta values (\cite{AND&THA2},\cite{THA2}).\par

%%%%%%%%%%%%%%%%%%%%%%%NOTATION%%%%%%%%%%%%%%%%%%%%%%%%%%%%%%%%%%%%%%%%%%%%%%%%%%%%%

\section{Notation and preliminaries}\label{Notation}${}$\par

Let $\mathbb F_q$ be a finite field having $q$ elements and of characteristic $p.$ Let $K/\mathbb F_q$ be a global function field ($\mathbb F_q$ is algebraically closed in $K$). Let $\infty$ be a place of $K.$  We denote by $K_\infty$ the $\infty$-adic completion of $K,$ and by $\mathbb F_\infty$ the residue field of $K_\infty.$ We fix $\overline{K}_\infty$ an algebraic closure of $K_\infty.$ Let $\overline{K}$ be the algebraic closure of $K$ in $\overline{K}_\infty,$ and let $\overline{\mathbb F}_q\subset \overline{K}$ be the algebraic closure of $\mathbb F_q$ in $\overline{K}_\infty.$ Let $\mathbb C_\infty$ be the completion of $\overline{K}_\infty.$ Let $A$ be the ring of elements of $K$ which are regular outside of $\infty.$ Let $v_\infty: K_\infty \rightarrow \mathbb Z\cup \{+\infty\}$ be the normalized discrete valuation associated to the local field $K_\infty.$ We still denote by $v_\infty: \mathbb C_\infty\rightarrow \mathbb R\cup\{+\infty\}$  the extension of $v_\infty$ to $\mathbb C_\infty.$ \par
Let $\mathcal  I(A)$ be the group of fractional ideals of $A,$ and $\mathcal P(A)=\{ \alpha A, \alpha \in K^\times\}.$   Recall that ${\rm Pic}(A)=\frac{\mathcal I(A)}{\mathcal P(A)}$ is a finite abelian group. For any ideal $I\subset A, I\not =\{0\},$ we set:
 $$\deg I=\dim_{\mathbb F_q}\frac{A}{I}.$$
 Note that this function  on non-zero ideals of $A$ extends naturally  to a morphism $\deg: \mathcal I(A)\rightarrow \mathbb Z.$
 In particular, we have:
 $$\forall \alpha \in K^\times, \quad \deg \alpha := \deg \alpha A= -d_\infty v_\infty (\alpha),$$
 where $d_\infty =\dim_{\mathbb F_q} \mathbb F_\infty.$\par
 ${}$\par
Let $x \in \overline{K}_\infty$ such that $v_\infty(x)>0.$ Let  $(x_n)_{n\geq 1}$ be a sequence of elements in $ \overline{K}_\infty$ such that $x_1=x,$ and for $n\geq 1, $ $x_{n+1}^{n+1}=x_n.$ Let $z\in \mathbb Q,$ write $z=\frac{m}{ n!}, m\in \mathbb Z, n\geq 1,$ we set:
 $$x^z= x_n ^{m}.$$
 This is well-defined. We set:
 $$U_\infty =\{ y\in \overline{K}_\infty, v_\infty (y-1)>0\}.$$
 Observe that $U_\infty$ is a $\mathbb Q_p$-vector space. We have:
 $$\overline{K}_\infty^\times= x^{\mathbb Q}\times \overline{\mathbb F}_q^\times \times U_\infty.$$
 Let $y\in \overline{K}_\infty^\times$, then $y$ can be written in a unique way:
 $$y=x^{\frac{v_\infty(y)}{v_\infty(x)}}\omega_x(y)\langle y \rangle _x, \quad \omega_x(y)\in \overline{\mathbb F}_q^\times, \langle y \rangle _x\in U_\infty.$$
 Observe that the map $\langle . \rangle _x: \overline{K}_\infty^\times \rightarrow U_\infty$ is a group homomorphism such that for $y\in U_\infty,$ $\langle y \rangle _x=y.$ We will need the following Lemma in the sequel (see also \cite{GOS4}, Lemma 2):
\begin{lemma}\label{LemmaS2.1} ${}$\par
\noindent 1) Let $\langle . \rangle ': \overline{K}_\infty^\times\rightarrow U_\infty$ be a group homomorphism such that for $\forall y\in U_\infty, \langle y \rangle '=y.$ Then there exists a unique element $u\in U_\infty$ such that:
$$\forall y\in \overline{K}_\infty^\times, \quad \langle y \rangle '= \langle y \rangle _xu^{v_\infty(y)}.$$
\noindent 2) Let $\sgn:K_\infty^\times \rightarrow \overline{\mathbb F}_q^\times$ be a group homomorphism such that $\forall \zeta \in \mathbb F_q^\times, \sgn (\zeta)=\zeta.$ Then there exist $e\in \mathbb N,$ $e\equiv 1\pmod{q-1},$  and an  element $\zeta'\in \overline{\mathbb F}_q^\times,$ such that:
$$\forall y\in K_\infty^\times, \quad \sgn (y) = \omega_x(y)^e \zeta'^{v_\infty(y)}.$$
3) For $i=1,2,$ let $\langle . \rangle _i: \overline{K}_\infty^\times\rightarrow U_\infty$ be a group homomorphism such that $\forall y\in U_\infty, \langle y \rangle _i=y,$ and let $\gamma_i\in \overline{K}_\infty^\times.$ Let's assume that there exist an open subgroup of finite index $N\subset K_\infty^\times,$ and $n_1, n_2\in \mathbb Z_p,$ such that:
$$\forall a\in N\cap K^\times, \quad \langle a \rangle _1^{n_1}\gamma_1^{v_\infty(a)}=\langle a \rangle _2^{n_2}\gamma_2^{v_\infty(a)}.$$
Then $n_1=n_2.$
\end{lemma}
\begin{proof}${}$\par
\noindent 1)   We have:
$$u=(\langle x \rangle ')^{\frac{1}{v_\infty(x)}} .$$
\noindent 2) There exists $e\in \mathbb N, $ $e\equiv 1\pmod{q-1}$ such that:
$$\forall \zeta \in \mathbb F_\infty^\times, \quad \sgn (\zeta)= \zeta^e.$$
Let $U= K_\infty^\times \cap U_\infty.$ We get:
$$\frac{\sgn(.)}{\omega_x(.)^e}\mid_{\mathbb F_\infty^\times\times U}=1.$$
Let $\pi$ be any element of $K_\infty^\times$ such that $v_\infty(\pi)=1,$  and set:
$$\zeta'= \frac{\sgn (\pi)}{\omega_x(\pi)^e} \in \overline{\mathbb F}_q^\times.$$
3) Let $a\in U_\infty \cap N\cap K^\times,$ $a\not =1.$ Then:
$$a^{n_1}=a^{n_2}.$$
Thus $n_1=n_2.$

%By 1), there exists $u\in U_\infty$ such that:
%$$\forall a\in N, \quad \langle a \rangle _x^{n_1-n_2}=u^{v_\infty(a)}.$$
%Let $\pi\in K_\infty,$ $v_\infty(\pi)=1.$ Let $b\in \mathbb N\setminus\{0\}$ such that $\pi^b \in N.$ Select $\alpha \in U_\infty \cap N$ such that $\alpha \not =1.$ Then, we get:
%$$\forall n\geq 0, \quad \alpha ^{n_1-n_2}= \Big(\frac{u^b}{\langle \pi \rangle _x^{b(n_1-n_2)}}\Big)^{p^n}.$$
%Thus $\alpha ^{n_1-n_2}=1,$ and therefore $n_1=n_2.$
\end{proof}
${}$\par
%Recall that we have  fixed a $K$-embedding of $\overline{K}$ in $\overline{K}_\infty.$
A map $\eta: \mathcal I(A) \rightarrow \overline{K}_\infty$ will be called \textit{admissible} if there exist  an open subgroup of finite index $N \subset K_\infty^\times,$ an element $n(\eta)\in \mathbb Z_p,$   an element $\gamma \in \overline{K}_\infty^\times,$
and a group homomorphism $\langle . \rangle : \overline{K}_\infty \rightarrow U_\infty, $ with $\langle . \rangle \mid_{U_\infty}= {\rm Id}_{U_\infty},$  such that:
$$\forall I\in \mathcal I(A), \forall \alpha \in K^\times \cap N, \quad \eta (I\alpha)= \eta (I) \langle \alpha \rangle ^{n(\eta)} \gamma^{v_\infty(\alpha)}.$$
By Lemma \ref{LemmaS2.1}, this definition does not depend on the choice of $\langle . \rangle $, $n(\eta)$ is well-defined,  and $\gamma$ is well-defined modulo $U_\infty ^{n(\eta)}.$ Note also that the product of admissible maps is an admissible map.\par
We  observe that any constant  map from  $ \mathcal I(A)$ to  $\overline{K}_\infty$ is admissible. An admissible map $\eta$ will be called algebraic if $\eta(\mathcal I(A))\subset \overline{K}.$\par
${}$ \par
Let's give a fundamental example, due to D. Goss, of such an admissible map (\cite{GOS}, section 8.2). Let $t$ be a prime of $K_\infty$ (i.e. $v_\infty (t)=1$). For $x\in K_\infty^\times,$ let's write:
$$x=t^{v_\infty(x)} \omega_t(x) \langle x \rangle _t, \quad \omega_t(x)\in \mathbb F_\infty^\times, v_\infty( \langle x \rangle _t-1) > 0.$$
Write $d_\infty= p^k \bar{d}_\infty$ with $\bar {d}_\infty\not \equiv 0\pmod{p}.$ Let $p^m$ be the biggest power of $p$ that divides $\mid {\rm Pic}(A)\mid.$ Set $l={ \rm Max}\{ m,k\}.$  We fix $t^{\frac{1}{\bar {d}_\infty}}\in \overline{K}_\infty$ a $\bar{d}_\infty$-th root of $t.$ Let:
$$T=\mathbb F_\infty((t^{\frac{1}{p^l\bar{d}_\infty}})).$$
If $I\in \mathcal I(A),$ then there exists an integer $h$ dividing $\mid {\rm Pic}(A)\mid,$ such that:
$$I^h= \alpha A, \quad \alpha \in K^\times.$$
We set:
$$[I]_t= (\langle \alpha  \rangle _t)^{\frac{1}{h}} t^{\frac{-deg I}{d_\infty}} \in T.$$
Then, $K([I], I\in \mathcal I(A))/K$ is a finite extension, $[.]_t: \mathcal I(A)\rightarrow \overline{K}_\infty^\times$ is a group homomorphism and:
$$\forall \alpha \in K^\times, \quad [\alpha A]_t=\frac{\alpha }{\omega_t(\alpha)},$$
Thus, by Lemma \ref{LemmaS2.1},  $[.]_t$ is an algebraic  admissible map.\par
${}$\par
From now in, we fix a prime $\pi$ of $K_\infty$ ($\pi \in K_\infty$, $v_\infty (\pi)=1$). We also fix a sequence $(\pi_n)_{n\geq 1}$  of elements in $ \overline{K}_\infty$ such that $\pi_1=\pi,$ and $\forall n\geq 1, $ $\pi_{n+1}^{n+1}=\pi_n.$  By our previous construction, attached to this choice of such a sequence, we have a group homomorphism  $\langle . \rangle =\langle . \rangle _\pi: \overline{K}_\infty^\times \rightarrow U_\infty,$ and a group homomorphism $\sgn = \omega_\pi: \overline{K}_\infty^\times \rightarrow \overline{\mathbb F}_q^\times.$
For $I\in \mathcal I(A),$ we set:
$$[I]= \langle I \rangle  \pi^{-\frac{\deg I}{d_\infty}}\in \pi^{\mathbb Q} \times U_\infty,$$
where:
$$\langle I \rangle =\langle \alpha  \rangle ^{\frac{1}{h}}, \quad I^h=\alpha A, \quad \alpha \in K^\times.$$
Then, $[.]: \mathcal I(A)\rightarrow  \overline{K}_\infty^\times$ is a group homomorphism such that:
$$\forall \alpha \in K^\times, \quad [\alpha A]= \frac{\alpha}{\sgn (\alpha)}.$$
Observe  that $[.]$ is an algebraic admissible  map, and $\langle . \rangle : \mathcal I(A)\rightarrow U_\infty$ is an admissible map.\par
${}$\par
Let $\eta: \mathcal I(A) \rightarrow \overline{K}_\infty$ be an admissible map. Then there exist $M(\eta)\subset K_\infty^\times$ an open subgroup of finite index, an element $n(\eta)\in \mathbb Z_p,$ an element  $\gamma_\eta\in \overline{K}_\infty^\times ,$ such that:
$$\forall I\in \mathcal I(A), \forall \alpha \in K^\times \cap M(\eta), \quad \eta(I\alpha)= \eta (I) (\frac{\alpha}{\sgn (\alpha)\pi^{v_\infty(\alpha)}})^{n(\eta)} \gamma_\eta^{v_\infty(\alpha)}.$$
Let's observe that, if $n(\eta)\in \mathbb Z,$ $\forall \alpha \in K^\times \cap M(\eta),$ we have:
$$ (\frac{\alpha}{\sgn (\alpha)\pi^{v_\infty(\alpha)}})^{n(\eta)} \gamma_\eta^{v_\infty(\alpha)}= (\frac{\alpha}{\sgn (\alpha)})^{n(\eta)} (\gamma_\eta \pi^{-n(\eta)})^{v_\infty(\alpha)}.$$
In that case, we set $K(\eta):= K(\mathbb F_\infty,\gamma_\eta \pi^{-n(\eta)}, \eta (I), I\in \mathcal I(A)).$ Let's observe that, if $\eta$ is algebraic then $K(\eta)/K$ is a finite extension.\par

 ${}$\par
 Let $(F,v_F)$ be a field which is complete with respect to a non-trivial valuation $v_F: F\rightarrow \mathbb R\cup\{+\infty\}$ and such that $F$ is a $\overline{\mathbb F}_q$-algebra.  \par
Let $n\geq 1,$ we set:
 $$\mathbb S_{F,n}=F^\times \times \mathbb Z_p^n.$$
 We view $\mathbb S_{F,n}$ as an abelian group with group action  written additively.  Let $\eta_1, \ldots, \eta_n: \mathcal I(A) \rightarrow \overline{K}_\infty^\times$ be $n$ admissible maps (note that for $i=1, \ldots, n,$  $\forall I\in \mathcal I(A), $ $\eta_i(I)\not =0$).
Observe that for $i=1, \ldots, n,$ $K_\infty(\langle \eta_i \rangle ):= K_\infty( \langle \gamma_{\eta_i} \rangle , \langle \eta_i(I) \rangle , I\in \mathcal I(A))$ is a finite extension of $K_\infty.$
For $i=1, \ldots, n,$ let $\sigma_i: K_\infty (\langle \eta_i \rangle )\hookrightarrow F$ be a continuous $\mathbb F_p$-algebra homomorphism.
For $u=(x; z_1, \ldots, z_n)\in \mathbb S_{F,n}$ and $I\in \mathcal I(A),$  we set:
 $$I_{\underline {\sigma}, \underline{\eta}}^{\underline{z}}=  \prod_{i=1}^n \sigma_i(\langle \eta_i(I) \rangle ^{z_i})\in F^\times,$$
 $$I_{\underline {\sigma}, \underline{\eta}}^{u}=  x^{\deg I}\prod_{i=1}^n \sigma_i(\langle \eta_i(I) \rangle ^{z_i})\in F^\times.$$
 Observe that for $I\in \mathcal I(A)$,  $u,u' \in \mathbb S_{F,n}$ and $\alpha \in K^\times\cap M(\eta_1)\cap \cdots \cap M(\eta_n)$, we have:
 $$I_{\underline{\sigma}, \underline{\eta}}^{u+u'}=I_{\underline{\sigma}, \underline{\eta}}^{u}I_{\underline{\sigma}, \underline{\eta}}^{u'},$$
 $$(I\alpha)_{\underline{\sigma}, \underline{\eta}}^{u}= I_{\underline{\sigma}, \underline{\eta}}^{u} (\alpha A)_{\underline{\sigma},\underline{\eta}}^{u},$$
 where:
 $$(\alpha A)_{\underline{\sigma}, \underline{\eta}}^{u}:=(x^{-d_\infty}\prod_{i=1}^n \sigma_i(\langle \gamma_{\eta_i} \rangle ^{z_i}) )^{v_\infty(\alpha)}\prod_{i=1}^n \sigma_i(\langle \alpha \rangle ^{n(\eta_i)z_i})\in F^\times.$$

\begin{remark} \label{Remark S2.1}
Let $N\subset K_\infty^\times$ be an open subgroup of finite index, and let's set:
$$\mathcal P(N)=\{ xA, x\in K^\times \cap N\}.$$
Then $\frac{\mathcal I(A)}{\mathcal P(N)}$ is a finite abelian group. Let $\chi: \mathcal I(A)\rightarrow \overline{K}_\infty$ be a map such that:
$$\forall I\in \mathcal I(A), \forall J\in \mathcal P(N), \quad \chi (IJ)= \chi (I).$$
Observe that $\chi$ is an admissible map.
Let $\gamma' \in \overline{K}_\infty^\times,$ and let $z\in \mathbb Z_p.$
Then $$\eta= \chi(.)\langle . \rangle ^z (\gamma') ^{\deg (.)}$$ is an admissible map with $M(\eta)=N,$  $n(\eta)=z, $ and $\gamma_\eta =(\gamma')^{-d_\infty}$.\par
Reciprocally, let $\eta$ be an admissible map. Select $\gamma'\in \overline{K}_\infty^\times$ such that $(\gamma')^{d_\infty}= \frac{1}{\gamma_\eta}.$ Set $\chi (.)=\frac{\eta(.)}{\langle . \rangle ^{n(\eta)}(\gamma')^{\deg(.)}}.$ Then $\chi: \mathcal I(A) \rightarrow \overline{K}_\infty$ is such that:
$$\forall I\in \mathcal I(A), \forall J\in \mathcal P(M(\eta)), \quad \chi (IJ)= \chi (I).$$
\end{remark}
\begin{remark}\label{RemarkS2.2} Let $T/K_\infty$ be a finite extension, let $\mathbb F_T$ be the residue field of $T,$ and let $\pi'$ be a prime of $T.$ Let $\sigma : T\hookrightarrow F$ be a continuous $\mathbb F_p$-algebra homomorphism.
Let $\varphi \in {\rm Gal} (T/\mathbb F_p((\pi')))\simeq {\rm Gal}(\mathbb F_T/\mathbb F_p),$ such that:
$$\forall \zeta \in \mathbb F_T, \quad \sigma (\zeta) =\varphi(\zeta).$$
Let $y=\sigma (\pi').$ Then $v_F(y)>0,$ and:
$$\forall a\in T, \quad \sigma (a) =\varphi (a)\mid_{\pi'=y}.$$
Thus the choice of a continuous $\mathbb F_p$-algebra homomorphism $\sigma: T\hookrightarrow F$ is equivalent to the choice of an element $y\in F^\times,$ $v_F(y)>0,$ and   $\varphi \in {\rm Gal}(\mathbb F_T/\mathbb F_p).$  \end{remark}

%%%%%%%%%KEY LEMMA%%%%%%%%%%%%%%%%%%%%%%%%% %%%%%%%%%%%%%%%%%%%%%%%%%%%%%%%%%%%%%%%%%

\section{A Key Lemma}\label{KeyLemma}${}$\par
 Let $E/K$ be a finite extension, $E\subset \overline{K},$ let $g_E$ be the genus of $E,$  and let $O_E$ be the integral closure of $A$ in $E.$ Let $\mathcal I(O_E)$ be the group of fractional ideals of $O_E,$ and let $\mathcal P(O_E)=\{ xO_E, x\in  E^\times\}.$ Let $N_{E/K}: \mathcal I(O_E)\rightarrow \mathcal I(A)$ be the group homomorphism such that, if $\frak P$ is a maximal ideal of $O_E$ and if $P=\frak P\cap A$ is the corresponding maximal ideal of $A,$ then:
 $$N_{E/K} (\frak P)=P^{[\frac{O_E}{\frak P}:\frac{A}{P}]}.$$
 Note that, if $\frak I=xO_E, x\in E^\times,$ then:
 $$N_{E/K} (\frak I) =N_{E/K} (x)A,$$
 where $N_{E/K}: E\rightarrow K$ denotes also  the usual norm map. Let $\infty _1, \ldots , \infty_r$ be the places of $E$ above $\infty.$ For $i=1, \ldots, r,$ let  $f_i$ be the residual degree of $\infty_i/\infty.$  Let $v_1, \ldots, v_r$ be the valuations on $E$ associated to these places. We denote by $E_{v_i}$ the $v_i$-adic completion of $E,$ $i=1, \ldots, r.$ \par
 Let $\mathcal B\subset O_E,$ $\mathcal B\in \mathcal I(O_E).$ We denote by $\mathcal I(\mathcal B)$ the group of fractional ideals of $O_E$ which are relatively prime to $\mathcal B.$ For $i=1, \ldots, r,$ let $N_i$ be an open subgroup of finite index of $E_{v_i}^\times,$ and let $m_i\geq 0$ be the least integer such that:
 $$\{ x\in E_{v_i}^\times , v_i(x-1)\geq m_i\}\subset N_i.$$
 We set $N=\prod_{i=1}^rN_i,$ and:
 $$\mathcal P(\mathcal B, N)= \{ xO_E\in \mathcal I(\mathcal B), x\in E^\times, x\equiv 1\pmod{\mathcal B}, x\in N_i,  i=1, \ldots, r\}.$$
Observe that $\frac{\mathcal I(\mathcal B)}{\mathcal P(\mathcal B, N)}$ is a finite abelian group. Let:
$$d(\mathcal B)= \dim_{\mathbb F_q} \frac{O_E}{\mathcal B},$$
$$d_\infty(N)= \sum_{i=1}^r d_\infty f_i(m_i+1).$$
More generally, let $\mathcal S$ be a finite set, possibly empty,  of maximal ideals of $O_E$ which are relatively prime to $\mathcal B.$ We denote by  $\mathcal I_{\mathcal S}(\mathcal B)$ the group of fractional ideals of $\mathcal I(O_E)$   which are relatively prime to $\mathcal B\prod_{\frak P\in \mathcal S} \frak P.$ We also set $\mathcal P_{\mathcal S}(\mathcal B, N)= \mathcal P(\mathcal B, N)\cap \mathcal I_{\mathcal S}(\mathcal B).$ Observe that we have a natural group isomorphism:
$$ \frac{\mathcal I_{\mathcal S}(\mathcal B)}{\mathcal P_{\mathcal S}(\mathcal B, N)}\simeq \frac{\mathcal I(\mathcal B)}{\mathcal P(\mathcal B, N)} .$$
We set:
$$d_{\mathcal S}(\mathcal B)=\dim_{\mathbb F_q}\frac{O_E}{\mathcal B\prod_{\frak P\in \mathcal S} \frak P}.$$
$ {}$\par
We recall a basic result (\cite{GOS}, Lemma 8.8.1):
 \begin{lemma}\label{LemmaS3.1}
 Let $ r\in \mathbb N ,$ and let $W$ be a  finite dimensional $\mathbb F_p$-vector space. For $i=1, \ldots, r,$ let $f_i: W\rightarrow F$ be an $\mathbb F_p$-linear map. If $\dim_{\mathbb F_p} W  > \frac{r}{p-1},$ then:
 $$\forall x_1, \ldots, x_r \in F, \qquad \sum_{w\in W} (x_1+f_1(w))\cdots (x_r+f_r(w))=0.$$
 \end{lemma}
 \begin{proof} We recall its proof for the convenience of the reader. First, note that for  $k, m_1, \ldots m_k\in \mathbb N ,$ $k\geq 1,$ then :
 $$\sum_{\zeta_1, \ldots , \zeta_k \in \mathbb F_p}\zeta_1^{m_1}
\cdots \zeta_k^{m_k}\not =0 \Leftrightarrow \forall i\in \{1, \ldots k\}, \quad m_i \equiv 0\pmod{p-1}\, {\rm and}\, m_i\geq 1 .$$
Thus, in particular, if $\sum_{\zeta_1, \ldots , \zeta_k \in \mathbb F_p}\zeta_1^{m_1}\cdots \zeta_k^{m_k}\not =0 $ then $m_1+\cdots +m_k \geq k(p-1).$ Let $\alpha =\dim_{\mathbb F_p} W.$ Let's assume that $\alpha >\frac{r}{p-1}.$  Let $$S= \sum_{w\in W} (x_1+f_1(w))\cdots (x_r+f_r(w)).$$  Let's set:
$$W=\oplus_{j=1}^{\alpha} \mathbb F_p e_j.$$  Then:
 $$S= \sum_{\zeta_1, \ldots , \zeta_\alpha  \in \mathbb F_p}\prod_{j=1}^r (x_j+\zeta_1 f_j(e_1)+\cdots + \zeta_\alpha f_j(e_\alpha)).$$
When we develop the above product and sum over $\zeta_1,\ldots, \zeta_\alpha \in \mathbb F_p,$ the sums that appear are of the form:
$$\sum_{ \zeta_1, \ldots , \zeta_\alpha  \in \mathbb F_p}\zeta_1^{i_1}\cdots \zeta_{\alpha}^{i_\alpha},$$
with $i_1, \ldots , i_\alpha \in \mathbb N,$ and $i_1+\cdots +i_\alpha\leq r.$ Since $r<(p-1)\alpha,$ these sums vanish and therefore $S=0.$
 \end{proof}
 Let $s\geq 0$ be an integer. For $i=1, \ldots, s,$ let $\sgn_i:K_\infty^\times \rightarrow \overline{\mathbb F}_q^\times$ be a group homomorphism such that $\forall \zeta \in \mathbb F_q^\times, \sgn_i(\zeta)=\zeta.$ Observe that:
$$\forall i\in \{1, \ldots, s\}, \forall x\in K_\infty^\times \cap U_\infty, \quad \sgn_i(x)=1.$$
For $i=1, \ldots, s,$  let $\rho_i: K(\sgn_i(K^\times))\hookrightarrow F$ be an $\mathbb F_p$-algebra homomorphism.\par
Lemma \ref{LemmaS3.1} combined with the Riemann-Roch Theorem yields the   following result which will be crucial in the sequel.

\begin{lemma}\label{LemmaS3.2}
Let $\eta$ be a generator of the cyclic group  $\mathbb F_\infty^\times.$ Let $\frak I$ be a non-zero ideal of $O_E$ with  $\frak I\in \mathcal I_{\mathcal S}(\mathcal B).$ Let $\mathbb F_E$ be the algebraic closure of $\mathbb F_q$ in $E.$  Let $k$ be an integer, $1\leq k\leq \frac{q^{d_\infty}-1}{q-1}.$  Let
\begin{align*}
S_{m, \mathcal S}(\frak I, \mathcal B, N,k)=\{ &aO_E\in \mathcal P_{\mathcal S}(\mathcal B, N), \sgn(N_{E/K}(a))\equiv \eta^k \pmod{\mathbb F_q^\times}; \\
 & \qquad a\frak I\subset O_E, \deg (N_{E/K}(a\frak I))=m \}.
\end{align*}
If $$m \geq  2g_E [\mathbb F_E:\mathbb F_q]+d_\infty(N)+d_{\mathcal S}(\mathcal B)+\frac{s [E:K]_{sep}}{(p-1)[\mathbb F_q:\mathbb F_p]}$$  where $[E:K]_{sep}$ is the separable degree of the finite extension $E/K$, then we have:
$$ \sum_{aO_E\in S_{m, \mathcal S}(\frak I, \mathcal B, N,k)}\rho_1(\frac{N_{E/K}(a)}{\sgn_1(N_{E/K}(a))})\cdots \rho_s(\frac{N_{E/K}(a)}{\sgn_s(N_{E/K}(a))})=0.$$ ${}$\par
In particular, let's set:
\begin{align*}
S_{m, \mathcal S}(\frak I, \mathcal B, N)&=\cup_{k=1}^{\frac{q^{d_\infty}-1}{q-1}}S_{m, \mathcal S}(\frak I, \mathcal B, N,k)\\
&=\{ aO_E\in \mathcal P_{\mathcal S}(\mathcal B, N), a\frak I\subset O_E, \deg (N_{E/K}(a\frak I))=m\}.
\end{align*}
If $$m \geq  2g_E [\mathbb F_E:\mathbb F_q]+d_\infty(N)+d_{\mathcal S}(\mathcal B)+\frac{s [E:K]_{sep}}{(p-1)[\mathbb F_q:\mathbb F_p]},$$ then we have:
$$ \sum_{aO_E\in S_{m, \mathcal S}(\frak I, \mathcal B, N)}\rho_1(\frac{N_{E/K}(a)}{\sgn_1(N_{E/K}(a))})\cdots \rho_s(\frac{N_{E/K}(a)}{\sgn_s(N_{E/K}(a))})=0.$$
 \end{lemma}
 \begin{proof}
The proof of the Lemma is based on the arguments used in the proof of  \cite{GOS}, Theorem 8.9.2. We give a detailed proof for the convenience of the reader.  Since there is a finite number of ideals $J$ of $A$ such that $\deg J=m,$ and, given such an ideal $J,$ there is a finite number of ideals $\frak J$ in $O_E$ such that $N_{E/K} (\frak J)=J,$ we deduce that $S_{m, \mathcal S}(\frak I, \mathcal B, N,k)$ is a finite set, possibly empty.  Let's fix $aO_E\in S_{m, \mathcal S}(\frak I,\mathcal B, N,k)$ (so we assume that $m\geq 0$ is such that $S_{m, \mathcal S}(\frak I, \mathcal B,N,k)\not =\emptyset$). Let $V(a)$  be the set of elements $b\in E$ such that

i) $b\frak I\subset O_E,$

ii) $v_i(b)\geq v_i(a)+m_i+1$, for $i=1, \ldots, r,$

iii) ${\rm ord}_{\frak P} (b)\geq {\rm ord}_{\frak P}(\mathcal B),$ for all $\frak P$ dividing $\mathcal B,$

iv) ${\rm ord}_{\frak P} (b)\geq 1,$  for all $\frak P\in \mathcal S$.

Observe that:
 $$\deg N_{E/K} (a\frak I)= \deg N_{E/K} (\frak I)-d_\infty \sum_{i=1}^r  f_iv_i(a).$$
 But:
 $$\deg N_{E/K} (\frak I)= [\mathbb F_E:\mathbb F_q] \dim_{\mathbb F_E }\frac {O_E}{\frak I},$$
 therefore:
 $$\frac{-1}{[\mathbb F_E:\mathbb F_q]}\sum_{i=1}^r d_\infty f_i v_i (a) =\frac{m}{[\mathbb F_E:\mathbb F_q]}-\dim_{\mathbb F_E }\frac {O_E}{\frak I}.$$

We now assume that ${m}-{d_\infty(N)-d_{\mathcal S}(\mathcal B)}\geq [\mathbb F_E:\mathbb F_q](2g_E-1).$ Then, by the Riemann-Roch Theorem, $V(a)$ is a finite $\mathbb F_E$-vector space of dimension $$\dim V(a)=\frac{m}{[\mathbb F_E:\mathbb F_q]}+1-g_E-\frac{d_\infty (N)+d_{\mathcal S}(\mathcal B)}{[\mathbb F_E:\mathbb F_q]}.$$   We have (see \cite{GOS}, Lemma 8.9.3):
 $$ (a+V(a))O_E\subset S_{m, \mathcal S}(\frak I, \mathcal B, N,k),$$
 $$\forall b, b'\in V(a), \quad (a+b)O_E= (a+b')O_E\Leftrightarrow b=b'.$$
 If $\sgn': K_\infty^\times \rightarrow \overline{\mathbb F_q}^\times$ is a group homomorphism, we have:
 $$ \forall b\in V(a), \quad  \sgn'(N_{E/K}(a+b))=\sgn'(N_{E/K}(a)).$$
 Let $cO_E\in S_{m, \mathcal S}(\frak I, \mathcal B, N,k).$  Then:
 $$(c+V(c))O_E\cap (a+V(a))O_E\not = \emptyset \Leftrightarrow cO_E\in (a+V(a))O_E.$$
 Thus, if $S_{m, \mathcal S}(\frak I, \mathcal B,N,k)\not =\emptyset,$ let's select  $a_1O_E, \ldots, a_lO_E\in S_{m, \mathcal S}(\frak I,\mathcal B, N,k)$ such that $S_{m, \mathcal S}(\frak I, \mathcal B, N,k)$ is the disjoint union of the $(a_i+V(a_i))O_E.$ We have:
 \begin{align*}\sum_{cO_E\in S_{m, \mathcal S}(\frak I, \mathcal B, N,k)}\rho_1(\frac{N_{E/K}(c)}{\sgn_1(N_{E/K}(c))})\cdots \rho_s(\frac{N_{E/K}(c)}{\sgn_s(N_{E/K}(c))})= \cr
\sum_{i=1}^l\sum_{b\in V(a_i)}\rho_1(\frac{N_{E/K}(a_i+b)}{\sgn_1(N_{E/K}(a_i))})\cdots \rho_s(\frac{N_{E/K}(a_i+b)}{\sgn_s(N_{E/K}(a_i))}).
\end{align*}
 Let's fix $1\leq i\leq l,$ and let $$S= \sum_{b\in V(a_i)}\rho_1(N_{E/K}(a_i+b))\cdots \rho_s(N_{E/K}(a_i+b)).$$ Let $E_1/K$ consist of the elements of $E$ which are separable over $K.$ Let $p^{l'}=[E:E_1].$ Then:
 $$S=\sum_{b\in V(a_i)}\rho_1(N_{E_1/K}(a_i^{p^{l'}}+b^{p^{l'}}))\cdots \rho_s(N_{E_1/K}(a_i^{p^{l'}}+b^{p^{l'}})) .$$
 Therefore, we can assume that $E/K$ is a finite separable extension.   Let $\overline{F}$ be an algebraic closure of $F.$  Then, for $i=1, \ldots, s,$  $\rho_i$ extends to a morphism $\rho_i: \overline{K}\rightarrow \overline{F}.$  Let $\sigma_l:E\rightarrow \overline{K}$ be the distincts $K$-embeddings of $E$ in $\overline{K},$ $l=1, \ldots, [E:K].$ Then:
$$\forall i=1, \ldots, s, \quad \rho_i(N_{E/K} (a_i+b))=\prod_{l=1}^{[E:K]}\rho_i(\sigma_l (a_i+b)).$$
By Lemma \ref{LemmaS3.1}, if  $[\mathbb F_q:\mathbb F_p](m +[\mathbb F_E:\mathbb F_q] (1-g_E)-d_\infty(N)-d_{\mathcal S}(\mathcal B)) > \frac{s [E:K]}{p-1},$ we have $S=0.$
 \end{proof}
Let $s\geq 0$ be an integer. For $i=1, \ldots, s,$ let $\psi_i: \mathcal I(A)\rightarrow \overline{K}_\infty$ be a map such that there exists  $M(\psi_i)\subset K_\infty^\times$  an open subgroup of finite index with the following property:
$$\forall I\in \mathcal I(A), \forall \alpha \in K^\times \cap M(\psi_i), \quad \psi_i (I\alpha )= \psi_i(I) \frac{\alpha}{\sgn_i (\alpha)}.$$
 For $i=1, \ldots, s,$ let $K(\psi_i)=K(\sgn_i(K^\times), \psi_i(I), I\in \mathcal I(A))\subset \overline{K}_\infty.$ %Observe that  $K(\psi_i)/K$ is a finite extension if $\psi_i$ is an algebraic admissible map.\par
 Let $\chi: \mathcal I(\mathcal B)\rightarrow F$ be a map such that:
 $$\forall \frak I\in I(\mathcal B), \forall \frak J\in \mathcal P(\mathcal B, N), \quad \chi (\frak I \frak J)= \chi (\frak I) .$$
 For $j=1, \ldots, r,$ let $N'_j =N_j\cap N_{E_{v_j}/K_\infty}^{-1}(M(\psi_1)\cap \ldots \cap M(\psi_s)),$ then $N'_j$ is an open subgroup of finite index of $E_{v_j}^\times.$ Let $N'= \prod_{j=1}^r N'_j.$
 \begin{corollary}\label{CorollaryS3.1} Let $s\geq 0$ be an integer and let, for $i=1, \ldots, s,$ $\rho_i:K(\psi_i)\hookrightarrow F$ be an $\mathbb F_p$-algebra homomorphism.  If $$m \geq  2g_E [\mathbb F_E:\mathbb F_q]+d_\infty(N')+d_{\mathcal S}(\mathcal B)+\frac{s [E:K]_{sep}}{(p-1)[\mathbb F_q:\mathbb F_p]},$$ then:
 $$ \sum_{\substack{\frak I\subset O_E, \frak I\in \mathcal I_{\mathcal S}(\mathcal B) \\
                \deg N_{E/K}(\frak I)=m}}
        \chi (\frak I)\, \rho_1(\psi_1(N_{E/K}(\frak I)))\cdots \rho_s(\psi_s(N_{E/K}(\frak I)))=0.$$
 \end{corollary}
  \begin{proof} For $m\geq 0,$ let's set:
 $$U_m=\sum_{\substack{\frak I\subset O_E, \frak I\in \mathcal I_{\mathcal S}(\mathcal B) \\
                \deg N_{E/K}(\frak I)=m}}
        \chi (\frak I)\, \rho_1(\psi_1(N_{E/K}(\frak I)))\cdots \rho_s(\psi_s(N_{E/K}(\frak I))).$$
 Let $h=\mid \frac{\mathcal I(\mathcal B)}{\mathcal P(\mathcal B, N')} \mid.$ Select $\frak I_1, \ldots, \frak I_h\in  \mathcal I_{\mathcal S}(\mathcal B)$, $\frak I_1, \ldots ,\frak I_h \subset O_E,$  a system of representatives of $\frac{\mathcal I(\mathcal B)}{\mathcal P(\mathcal B, N')}.$ Then $U_m$ is equal to $$\sum_{j=1}^h\chi(\frak I_j) \prod^s_{i=1} \rho_i(\psi_i(N_{E/K}(\frak I_j))) \sum_{aO_E\in S_{m, \mathcal S}(\frak I_j, \mathcal B, N')} \prod^s_{i=1} \rho_i(\frac{N_{E/K}(a)}{\sgn_i(N_{E/K}(a))}).$$
 \noindent Therefore, by Lemma \ref{LemmaS3.2}, if $m \geq  2g_E [\mathbb F_E:\mathbb F_q]+d_\infty(N')+d_{\mathcal S}(\mathcal B)+\frac{s [E:K]_{sep}}{(p-1)[\mathbb F_q:\mathbb F_p]},$ we have:
 $$U_m=0.$$
 \end{proof}
Let $s\geq 1$ be an integer.   Let $f:\mathbb Z\rightarrow 1+(q-1)\mathbb Z$ be a map. For $i=1, \ldots, s,$  let $\psi_i: \mathcal I(A)\rightarrow \overline{K}_\infty$ be a map  such that :
$$\forall I\in \mathcal I(A), \forall \alpha \in K^\times, \quad \psi_i(I\alpha)= \psi_i(I) \frac{\alpha}{\sgn_i(\alpha)^{f(\deg I)}}.$$
Let's observe  that, by the proof of Lemma \ref{LemmaS2.1}, for $i=1, \ldots, s,$ there exist $e_i\in \mathbb Z,$ $e_i\equiv 1\pmod{q-1},$ and  $\zeta_i\in \overline{\mathbb F}_q^\times$ such that:
 $$\forall a\in K_\infty^\times, \quad \sgn_i(a) =\sgn(a)^{e_i} \zeta_i^{v_\infty(a)}.$$
 Let $\chi: \mathcal I(A)\rightarrow F$ be a map such that:
 $$\forall I\in \mathcal I(A), \forall \alpha \in K^\times, \quad \chi (I\alpha)= \chi (I).$$
 For $i=1, \ldots, s,$ let $\rho_i:K( \psi_i)\hookrightarrow F$ be an $\mathbb F_p$-algebra homomorphism. For $i=1, \ldots, s,$ let $l_i\in \mathbb N $ such that: $$\forall \alpha \in K^\times, \quad \rho_i(\sgn_i(\alpha)) =\sgn_i(\alpha)^{p^{l_i}}.$$

\begin{corollary}\label{CorollaryS3.2}
Let $s\geq 1.$ If $$p^{l_1}e_1+\cdots +p^{l_s}e_s\equiv 0\pmod{q^{d_\infty}-1},$$ and  $$\prod_{i=1}^s \zeta_i^{p^{l_i} }=1,$$  then :
 $$\sum_{m\geq 0} \sum_{\substack{I\in \mathcal I(A), I \subset A \\ \deg  I=m}} \chi (I)  \rho_1(\psi_1( I))\cdots \rho_s(\psi_s( I))=0.$$
\end{corollary}

\begin{proof} Here $E=K, $ $\mathcal S=\emptyset,$ $\mathcal B=A,$ $N=K_\infty^\times.$ For $m\geq 0,$ let's set:
 $$U_m=\sum_{m\geq 0} \sum_{\substack{I\in \mathcal I(A), I \subset A \\ \deg  I=m}} \chi (I)  \rho_1(\psi_1( I))\cdots \rho_s(\psi_s( I)).$$
 Let $h=\mid {\rm Pic}(A) \mid.$ Select $I_1, \ldots,  I_h\in  \mathcal I (A)$, $I_1, \ldots , I_h \subset A,$  a system of representatives of ${\rm Pic}(A).$ Then $U_m$ is equal to: $$ \sum_{j=1}^h\chi(I_j) \prod^s_{i=1} \rho_i(\psi_i( I_j))  \sum_{aA\in S_{m, \mathcal S}(I_j, A, N)} \prod^s_{i=1}  \rho_i(\frac{a}{\sgn_i (a)^{f(\deg I_j)}}).$$
 Now, assume that $p^{l_1}e_1+\cdots +p^{l_s}e_s\equiv 0\pmod{q^{d_\infty}-1},$ and  $\prod_{i=1}^s \zeta_i^{p^{l_i} }=1.$  For simplicity, we set: $$S_m(I_j)=S_{m, \mathcal S}(I_j, A, N)=\{ aA, a\in K^\times, a I_j\subset A, \deg (a I_j))=m\}.$$ Then $\sum_{m\geq 0} U_m$ is equal to the finite sum (Lemma \ref{LemmaS3.2}):  $$\sum_{j=1}^h\chi( I_j)\rho_1(\psi_1( I_j))\cdots \rho_s(\psi_s( I_j)) \sum_{m\geq 0} \sum_{aA\in S_m( I_j)} \rho_1(a)\ldots \rho_s(a).$$
 Let's select $d\geq 0$ such that:
 $$\sum_{m\geq 0} \sum_{aA\in S_m(I_j)} \rho_1(a)\ldots \rho_s( a)=\sum_{m= 0}^d \sum_{aA\in S_m(I_j)} \rho_1(a)\ldots \rho_s( a).$$
 Let $V_{d}=\{ a\in K^\times, aA\in S_m( I_j), m\leq d\} \cup \{0\}.$ Then $V_{d}$ is a finite dimensional  $\mathbb F_q$-vector space and:
 $$\sum_{m= 0}^d \sum_{aA\in S_m(I_j)} \rho_1(a)\ldots \rho_s( a)=-\sum_{a\in V_d} \rho_1(a)\ldots \rho_s( a).$$
 By Lemma \ref{LemmaS3.1}, if $\dim_{\mathbb F_q}V_{d}> \frac{s}{(p-1)[\mathbb F_q:\mathbb F_p]},$ then:
 $$\sum_{a\in V_d} \rho_1(a)\ldots \rho_s( a)=0.$$
 Now, the Riemann-Roch Theorem tells us that, for  $d\geq 2g_K-1,$ we have: $\dim_{\mathbb F_q} V_d=d+1-g_K.$ Thus:
 $$\sum_{m\geq 0} \sum_{aA\in S_m(I_j)} \rho_1(a)\ldots \rho_s( a)=0.$$
\end{proof}

%%%%%%%%%%%%%%Several Variable Zeta Functions%%%%%%%%%%%%%%%%%%%%%%%%%%%%%%%%%%%%%%%%%%%%%%%%%%%%

\section{Several variable twisted zeta functions}\label{SVTZF}${}$\par

\subsection{The $\infty$-adic case}\label{TZF}${}$\par
Let $E/K$ be a finite extension, and let $\mathcal B$ and $N$ as in the previous section. Let $\mathcal S$ be a finite set, possibly empty,  of maximal ideal of $O_E$ which are relatively prime to $\mathcal B.$ Let $\chi: \mathcal I(\mathcal B)\rightarrow F$ be a map such that:
$$\forall \frak I\in \mathcal I(\mathcal B), \forall \frak J\in \mathcal P(\mathcal B,N), \quad \chi (\frak I\frak J)=\chi (\frak I).$$
Let $s\geq 0, n\geq 1$ be two integers. Let $\psi_1, \ldots, \psi_s, \eta_1, \ldots \eta_n: \mathcal I(A)\rightarrow \overline{K}_\infty$ be $s+n$ admissible maps. We assume that for $j=1, \ldots, n, $ $\eta_j(\mathcal I(A))\subset \overline{K}_\infty^\times,$ and for $i=1, \ldots, s,$ $n(\psi_i)\in \mathbb N.$   For $i=1, \ldots, s,$ let $\rho_i: K(\psi_i)\hookrightarrow F$ be an $\mathbb F_p$-algebra homomorphism.  For $i=1, \ldots, n,$ let $\sigma_i: K_\infty(\langle \eta_i \rangle )\hookrightarrow F$ be a continuous $\mathbb F_p$-algebra homomorphism. Let $u=(x; z_1, \ldots, z_n) \in \mathbb S_{F,n}=F^\times \times \mathbb Z_p^n.$
\begin{theorem}\label{TheoremS4.1} The following  sum converges in $F:$\par
$$\sum_{m\geq 0} \sum_{\substack{\frak I\in \mathcal I_{\mathcal S}(\mathcal B), \frak I\subset O_E \\ \deg N_{E/K} (\frak I)=m}} \chi (\frak I)\rho_1(\psi_1(N_{E/K}(\frak I))) \cdots \rho_s(\psi_s(N_{E/K}(\frak I))) N_{E/K}(\frak I)_{\underline{\sigma}, \underline{\eta}}^{-u}.$$
\noindent We denote the above sum by $\zeta_{\mathcal S, O_E}(\underline{\rho},\underline{\psi}; \underline{\sigma}, \underline{\eta}; \chi; u).$
\end{theorem}

\begin{proof} The proof is in the spirit of the proof of \cite{AP}, Proposition 6.\par

Let $\infty _1, \ldots , \infty_r$ be the places of $E$ above $\infty.$  let $v_1, \ldots, v_r$ be the valuation on $E$ associated to these places. We denote by $E_{v_i}$ the $v_i$-adic completion of $E,$ $i=1, \ldots, r.$ Recall that $N=\prod_{i=1}^r N_i,$ where $N_i\subset E_{v_i}^\times$ is an open subgroup of finite index. For $i=1, \ldots, r,$ let:
$$N'_i=N_i\cap N_{E_{v_i}/K_\infty}^{-1} (M(\psi_1) \cap \ldots \cap M(\psi_s)\cap M(\eta_1)\ldots \cap M(\eta_n)).$$
Then $N'_i\subset E_{v_i}^\times$ is also an open subgroup of finite index. Let's set:
$$N'=\prod_{i=1}^r N'_i.$$\par
%Step 1
 Let $\frak I_1, \ldots, \frak I_h\in  \mathcal I_{\mathcal S}(\mathcal B)$, $\frak I_1, \ldots ,\frak I_h \subset O_E$ be a system of representatives of $\frac{\mathcal I_{\mathcal S}(\mathcal B)}{\mathcal P_{\mathcal S}(\mathcal B, N')}.$   Recall that $$S_{m, \mathcal S}(\frak I_j, \mathcal B, N')=\{ aO_E\in \mathcal P_{\mathcal S}(\mathcal B, N'), a\frak I_j\subset O_E, \deg (N_{E/K}(a\frak I_j))=m\}.$$ Let's define $V_{m,j}$ to be the following sum:
 \begin{align*}
 \sum_{aO_E\in S_{m, \mathcal S}(\frak I_j, \mathcal B, N')}&(\prod_{i=1}^s\rho_i(\gamma_{\psi_i}\pi^{-n(\psi_i)}))^{v_\infty(N_{E/K}(a))}\, \prod_{i=1}^s\rho_i(\frac{N_{E/K}(a)}{\sgn(N_{E/K}(a))})^{n(\psi_i)}\, \times \\
 & \times (\prod_{j=1}^n\sigma_j(\langle \gamma_{\eta_j} \rangle ^{-z_j}))^{v_\infty(N_{E/K}(a))}\, \prod_{j=1}^n \sigma_j(\langle N_{E/K}(a) \rangle )^{-z_jn(\eta_j)}.
 \end{align*}
Then:
\begin{align*}
x^m \sum_{\substack{\frak I\in \mathcal I_{\mathcal S}(\mathcal B), \frak I\subset O_E \\ \deg N_{E/K} (\frak I)=m}} \chi (\frak I)\rho_1(\psi_1(N_{E/K}(\frak I))) \cdots \rho_s(\psi_s(N_{E/K}(\frak I))) N_{E/K}(\frak I)_{\underline{\sigma}, \underline{\eta}}^{-u} \\
=\sum_{j=1}^h \chi (\frak I_j) \rho_1(\psi_1(N_{E/K}(\frak I_j))) \cdots \rho_s(\psi_s(N_{E/K}(\frak I_j))) N_{E/K}(\frak I_j)_{\underline{\sigma}, \underline{\eta}}^{-\underline{z}}V_{m,j}.
\end{align*}

${}$\par
%Step 2
Note that $v_F\circ \rho_i\mid_K$ is a valuation on $K,$ thus, it is equivalent to the trivial valuation or to the valuation attached to a place of $K.$  This implies that there exists  integers $D,D'\in \mathbb R,$ such that for all $m\geq 0$, for all $j\in \{1, \ldots, h\}$ and for all $aO_E\in S_{m, \mathcal S}(\frak I_j, \mathcal B, N')$, we have:
$$v_F(\prod_{i=1}^s\rho_i(N_{E/K}(a))^{n(\psi_i)})\geq -Dm-D'.$$ ${}$\par
%Step 3
If $l\in \mathbb N,$ we denote by $\ell_p(l)$ the sum of the digits of $l$ in base $p.$ By Lemma \ref{LemmaS3.2}, if $m_1, \ldots m_n\in \mathbb N$ are such that $m_1+\ldots +m_n \geq 1$ and if $$m\geq 2g_E [\mathbb F_E:\mathbb F_q]+d_\infty(N')+d_{\mathcal S}(\mathcal B)+\frac{(s+ \sum_{i=1}^n \ell_p (m_i)) [E:K]_{sep}}{(p-1)[\mathbb F_q:\mathbb F_p]},$$ then: $$\sum_{aO_E\in S_{m, \mathcal S}(\frak I_j, \mathcal B, N')} \prod^s_{i=1} \rho_i(\frac{N_{E/K}(a)}{\sgn(N_{E/K}(a))}) \prod^n_{j=1} \sigma_j(\langle N_{E/K}(a) \rangle )^{m_j}=0.$$ ${}$\par
%step 4
Now, let  $l\in \mathbb N,$  and select $k_1, \ldots, k_n \in \mathbb N\setminus\{0\}$ such that:\par

i) $\ell_p(k_i)\leq  (l+1)(p-1),$\par

ii) $z_in(\eta_i)+k_i\equiv 0\pmod{p^{l+1}\mathbb Z_p}.$\par

\noindent For example, write $-z_i=\sum_{j\geq 0} u_{i,j} p^{ j} ,$ $ u_{i,j} \in \{0, \ldots, p-1\},$ and set $k_i= \sum_{j=0}^l u_{i,j} p^{ j} +p^{ l+1}.$ Let $$C= [E:K]_{sep}(p-1)[\mathbb F_q:\mathbb F_p](2g_E [\mathbb F_E:\mathbb F_q]+d_\infty(N')+d_{\mathcal S}(\mathcal B) +(\sum_{i=1}^s n(\psi_i))+1).$$ Then, if $m\geq Cn(l+1),$ we have:\par
$$\sum_{aO_E\in S_{m, \mathcal S}(\frak I_j, \mathcal B, N')} \prod^s_{i=1} \rho_i(\frac{N_{E/K}(a)}{\sgn(N_{E/K}(a))})^{n(\psi_i)} \prod^n_{j=1} \sigma_j(\frac{N_{E/K}(a)}{\sgn(N_{E/K}(a))})^{k_j}=0.$$ \par
\noindent Thus, there exists $D''\in \mathbb R$ such that:
$$v_F(V_{m,j})\geq p^{l+1}{\rm Inf}\{ v_F(y_i), i=1, \ldots, n\}-D''m-D',$$
where $y_i=\sigma_i(\pi), i=1, \ldots, n.$ Note also that $v_F(y_i)>0, i=1, \ldots, n.$
Therefore, if $m\geq Cn,$  we get:
$$v_F(V_{m,j})\geq p^{ [\frac{m}{Cn}]}{\rm Inf}\{ v_F(y_i), i=1, \ldots, n\}-D''m-D'.$$
Therefore, for $m\geq Cn ,$ there exists $D'''\in \mathbb R$ such that  we have:
\begin{align*}
& v_F \Big(x^m \sum_{\substack{\frak I\in \mathcal I_{\mathcal S}(\mathcal B), \frak I\subset O_E \\ \deg N_{E/K} (\frak I)=m}} \chi (\frak I)\rho_1(\psi_1(N_{E/K}(\frak I))) \cdots \rho_s(\psi_s(N_{E/K}(\frak I))) N_{E/K}(\frak I)_{\underline{\sigma}, \underline{\eta}}^{-u} \Big) \\
& \geq p^{[\frac{m}{Cn}]}{\rm Inf}\{ v_F(y_i), i=1, \ldots, n\}-D''m-D'''.
\end{align*}
\end{proof}

\begin{remark}\label{RemarkS4.1} By the proof of Theorem \ref{TheoremS4.1},  the valuation $v_F$ of the coefficient of $x^{-m}$ in $\zeta_{\mathcal S, O_E}(\underline{\rho},\underline{\psi}; \underline{\sigma}, \underline{\eta}; \chi; u)$ grows exponentially in $m$ for large $m$. \end{remark}
 Assume that $\psi_1, \ldots, \psi_s, \eta_1, \ldots \eta_n,$ and $\chi$ are group homomorphisms, then note that if $u=(x; z_1, \ldots, z_n)$ with $v_F(x)<<0,$ we have:
$$\zeta_{\mathcal S, O_E}(\underline{\rho},\underline{\psi}; \underline{\sigma}, \underline{\eta}; \chi; u)=\prod_{\frak P} (1- \chi (\frak P) \rho_1(\psi_1(N_{E/K}(\frak P)))\cdots \rho_s(\psi_s(N_{E/K}(\frak P)))N_{E/K}(\frak P)_{\underline{\sigma}, \underline{\eta}} ^{-u})^{-1},$$
where $\frak P$ runs through the set of maximal ideals of $O_E$ that are contained in $\mathcal I_{\mathcal S}(\mathcal B).$\par
${}$\par
Let $f:\mathbb Z\rightarrow 1+(q-1)\mathbb Z$ be a map. For $i=1, \ldots, s,$ let $\sgn_i:K_\infty^\times \rightarrow \overline{\mathbb F}_q^\times$ be a group homomorphism such that $\sgn_i\mid_{\mathbb F_q^\times}= {\rm Id}_{\mathbb  F_q^\times}.$  For $i=1, \ldots, s,$ let $\psi_i: \mathcal I(A)\rightarrow \overline{K}_\infty$ be a map such that:
$$\forall I\in \mathcal I(A), \forall \alpha \in K^\times, \quad \psi_i(I\alpha)= \psi_i(I) \frac{\alpha}{\sgn_i(\alpha)^{f(\deg I)}}.$$
For $i=1, \ldots, s,$ let $\rho_i: K(\sgn_i(K^\times), \psi_i(I), I\in \mathcal I(A))\rightarrow F$ be an $\mathbb F_p$-algebra homomorphism. For $j=1, \ldots, n,$ let $f_j:\mathbb Z\rightarrow 1+(q-1)\mathbb Z$ be a map, and  let $\eta_j: \mathcal I(A)\rightarrow \overline{K}_\infty^\times$ be a map such that
$$\forall I\in \mathcal I(A), \forall \alpha \in K^\times, \quad \eta_j(I\alpha)= \eta_j(I) \frac{\alpha}{\sgn'_j(\alpha)^{f_j(\deg I)}},$$
where $\sgn_j':K_\infty^\times \rightarrow \overline{\mathbb F}_q^\times$ is a group homomorphism such that $\sgn_j'\mid_{\mathbb F_q^\times}= {\rm Id}_{\mathbb  F_q^\times}.$
For $j=1, \ldots, n,$ let $\sigma_j : K_\infty(\sgn_j'(K^\times), \pi^{\frac{1}{d_\infty}},\langle \eta_j(I) \rangle , I\in \mathcal I(A))\rightarrow F$ be a continuous $\mathbb F_p$-algebra homomorphism.

Let's define $l_1, \ldots, l_s, l'_1, \ldots, l'_n \in \mathbb N$ as follows:
\begin{align*}
\forall i=1, \ldots, s, \forall \alpha \in K^\times , &\quad \rho_i(\sgn_i(\alpha))= \sgn_i(\alpha)^{p^{l_i}},\\
\forall j=1, \ldots, n, \forall \zeta\in \mathbb F_\infty^\times , &\quad \sigma_j(\zeta )= \zeta^{p^{l'_j}}.
\end{align*}
For $i=1, \ldots, s,$ there exist $e_i\in \mathbb Z,$ $e_i\equiv 1\pmod{q-1},$ and  $\zeta_i\in \overline{\mathbb F}_q^\times$ such that:
 $$\forall a\in K_\infty^\times, \quad \sgn_i(a) =\sgn(a)^{e_i} \zeta_i^{v_\infty(a)}.$$
We assume that:
$$\forall a\in K^\times,\quad \prod_{i=1}^s(\frac{\sgn_i(a)}{\sgn(a)})^{e_ip^{l_i}} =1, \text{ and } \quad \prod_{i=1}^s\zeta_i^{p^{l_i}}=1.$$
Let $\chi: \mathcal I(A)\rightarrow F$ be a map such that:
 $$\forall I\in \mathcal I(A), \forall \alpha \in K^\times, \quad \chi (I\alpha)= \chi (I).$$
Note that $\psi_1, \ldots, \psi_s, \eta_1, \ldots, \eta_n$ are admissible maps.
\begin{corollary}\label{CorollaryS4.2}
We assume that $E=K,$ $\mathcal B=A,$ $N=K_\infty^\times,$  $\mathcal S=\emptyset.$  Let $m_1, \ldots, m_n \in \mathbb N$ such that $ p^{l_1}e_1 +\dots+p^{l_s}e_s+m_1p^{l'_1}+\ldots+m_np^{l'_n} \equiv 0\pmod{q^{d_\infty}-1},$  and $p^{l_1}e_1+\dots+p^{l_s}e_s+m_1p^{l'_1}+\ldots+m_np^{l'_n} \geq1.$ Then:
$$\zeta_{\mathcal S, A}(\underline{\rho}, \underline{\psi};\underline{\sigma}, \underline{\eta};  \chi; (\sigma_1(\pi^{\frac{1}{d_\infty}})^{m_1}\cdots \sigma_n(\pi^{\frac{1}{d_\infty}})^{m_n}; -m_1, \ldots, -m_n))=0.$$
\end{corollary}
\begin{proof}  For $i=1, \ldots, n,$ let $y_i=\sigma_i(\pi^{\frac{1}{d_\infty}}).$ Let $u=( y_1^{m_1}\cdots y_n^{m_n}; -m_1, \ldots, -m_n)\in \mathbb S_{F,n}.$ Let $h=\mid {\rm Pic}(A) \mid.$ Select $I_1, \ldots,  I_h\in  \mathcal I (A)$, $I_1, \ldots , I_h \subset A,$  a system of representatives of ${\rm Pic}(A).$  Let's set:
 $$V_{m,j}=\sum_{aA\in S_{m,\mathcal S}( I_j, A,N)} \prod^{s}_{i=1} \rho_i(\frac{a}{\sgn_i(a)^{f(\deg I_j)}}) (aA)_{\underline{\sigma}, \underline{\eta}}^{(m_1, \ldots, m_n)}.$$
 Then, we have:
 $$V_{m,j}=\sum_{aA\in S_{m,\mathcal S}( I_j, A,N)} \prod^{s}_{i=1} \frac{\rho_i(a)}{\sgn(a)^{p^{l_i}e_i}} \prod^{n}_{j=1} \sigma_j(\frac{a}{\sgn(a)\pi^{v_\infty(a)}})^{m_j} .$$
Since $ p^{l_1}q^{e_1} +\dots+p^{l_s}q^{e_s}+m_1p^{l'_1}+\ldots+m_np^{l'_n} \equiv 0\pmod{q^{d_\infty}-1}$ and $x= y_1^{m_1}\cdots y_n^{m_n}$, $$V_{m,j} x^{-m}=-(\prod_{j=1}^ny_j^{-m_j\deg  I_j})\sum_{\substack{a\in K^\times \\ aA\in S_{m, \mathcal S}(I_j, A,N) }} \rho_1(a)\ldots \rho_s(a)\sigma_1(a)^{m_1} \cdots \sigma_n(a)^{m_n}  .$$
The assertion is then  a consequence of the proof of   Corollary \ref{CorollaryS3.2}.
\end{proof}

%%%%%%%%%%%%%%%%%%%%%%%%%

\subsection{The case of finite places}\label{Finiteplaces}${}$\par

Let $v$ be a finite place of $K,$ i.e. $v: K\rightarrow \mathbb Z\cup \{+\infty\}$ is a  discrete valuation on $K$ such that there exists $a\in A$ with $v(a)=1.$ Let $P_v=\{ a\in A, v(a)>0\}$ be the maximal ideal of $A$ corresponding to the finite place $v.$ We denote by $K_v$ the $v$-adic completion of $K.$ We fix $\overline{K}_v$ an algebraic closure of $K_v.$ Then $v$ extends to a valuation $v : \overline{K}_v\rightarrow \mathbb Q\cup\{+\infty\}.$ We fix a $K$-morphism $\overline{K} \hookrightarrow \overline{K}_v$, and, by abuse of notation, we will identify the elements in $\overline{K}$ with their images in $\overline{K}_v.$ We fix $\pi_v$ a prime of $K_v,$ i.e. $\pi_v\in K_v$ and $v(\pi_v)=1.$ Then, as in section \ref{Notation}, $\forall x\in \overline{K}_v^\times,$ $x$ can be written in a unique way:
$$x= \pi_v^{v(x)} \omega_v(x)\langle x \rangle _v \text{ with} \quad \omega_v(x) \in \overline{\mathbb F}_q^\times, v(\langle x \rangle _v-1)>0.$$ \par
Let $\psi_1, \ldots, \psi_s, \eta_1, \ldots \eta_n$ be $s+n$  admissible maps. We assume that $n(\psi_i)\in \mathbb N,i=1, \ldots, s,$ and $ n(\eta_j)\in \mathbb Z, j=1, \ldots, n,$ and that our  $s+n$ admissible maps are algebraic.   We  assume that, for $j=1, \ldots, n,$ $\eta_j(\mathcal I(A))\subset \overline{K}_v^\times.$  For $i=1, \ldots, s,$ we recall that  $K(\psi_i)= K(\mathbb F_\infty,\gamma_{\psi_i}\pi^{-n(\psi_i)}, \psi_i(I), I\in \mathcal I(A))\subset \overline{K}_v.$ For $j=1, \ldots, n,$ we set $K_v(\langle \eta_j \rangle _v):= K_v( \langle \eta_j(I) \rangle _v, I\in \mathcal I(A))\subset \overline{K}_v.$ Observe that, for $j=1, \ldots, n,$ $K_v(\langle \eta_j \rangle _v)/K_v$ is a finite extension.\par
Let $(F,v_F)$ be a field which is complete with respect to a non-trivial valuation $v_F: F\rightarrow R\cup\{+\infty\}$ and such that $F$ is an $\overline{\mathbb F}_q$-algebra.  \par
Let $n\geq 1,$ and let's set:
$$\mathbb S_{v,F,n}= F^\times \times \mathbb Z_p^n \times {\mathbb Z}^n.$$
For $i=1, \ldots, n, $ let $\sigma_i:  K_v(\langle \eta_i \rangle _v)\hookrightarrow F$ be a continuous $\mathbb F_p$-algebra homomorphism.  For $u=(x; z_1, \ldots, z_n; \underline{\delta})\in \mathbb S_{v,F,n},$ $I\in \mathcal I(A),$ we set:
$$I_{\underline{\sigma}, \underline{\eta}}^u= x^{\deg I} \omega_v( \eta_1(I))^{\delta_1} \cdots \omega_v( \eta_n(I))^{\delta_n}  \sigma_1(\langle \eta_1(I) \rangle _v^{z_1})\cdots \sigma_n(\langle \eta_n(I) \rangle _v^{z_n})\in F^\times.$$
Let $E/K$ be a finite extension, and let $\mathcal B$ and $N$ as in the previous section. Let $\mathcal S$ be a finite set, possibly empty,  of maximal ideals of $O_E$ which are relatively prime to $\mathcal B.$ Let $\mathcal S_v$ be the union of $S$ and the maximal ideals of $O_E$ above $P_v$ and that do not divide $\mathcal B.$ Let $\chi: \mathcal I(\mathcal B)\rightarrow F$ be a map such that:
$$\forall \frak I\in \mathcal I(\mathcal B), \forall \frak J\in \mathcal P(\mathcal B,N), \quad \chi (\frak I \frak J)=\chi (\frak I).$$ For $i=1, \ldots, s,$ let  $\rho_i: K(\psi_i)\hookrightarrow F$ be an $\mathbb F_p$-algebra homomorphism. Let $u=(x; z_1, \ldots, z_n;\underline{\delta}) \in \mathbb S_{v,F,n}.$

\begin{theorem} \label{TheoremS4.2}
The following  sum converges in $F:$
$$\sum_{m\geq 0} \sum_{\substack{\frak I\in \mathcal I_{\mathcal S_v}(\mathcal B), \frak I\subset O_E \\ \deg (N_{E/K} (\frak I))=m}} \chi (\frak I) \rho_1(\psi_1(N_{E/K}(I)))\cdots \rho_s(\psi_s(N_{E/K}(I))) N_{E/K} (\frak I)_{\underline{\sigma}, \underline{\eta}}^{-u}.$$
\end{theorem}

\begin{proof} The proof of the Theorem is similar to the proof of Theorem \ref{TheoremS4.1}. We only give a sketch of the proof. Let $N'$ be as in the proof of Theorem \ref{TheoremS4.1}.\par
%Step 1
Let $\frak I_1, \ldots, \frak I_h\in  \mathcal I_{\mathcal S_v}(\mathcal B)$, $\frak I_1, \ldots ,\frak I_h \subset O_E$ be a system of representatives of $\frac{\mathcal I(\mathcal B)}{\mathcal P(\mathcal B,N')}.$ Recall that $$S_{m, \mathcal S_v}(\frak I_j, \mathcal B, N')=\{ aO_E\in \mathcal P_{\mathcal S_v}(\mathcal B, N'), a\frak I_j\subset O_E, \deg (N_{E/K}(a\frak I_j))=m\}.$$ Let $V_{m,j}$ be the following sum:
\begin{align*}
& \sum_{aO_E\in S_{m, \mathcal S_v}(\frak I_j, \mathcal B, N')}(\prod_{i=1}^s\rho_i(\gamma_{\psi_i}\pi^{-n(\psi_i)}))^{v_\infty(N_{E/K}(a))}\omega_v(\prod_{j=1}^n\sigma_j(\gamma_{\psi_j}\pi^{-n(\eta_j)})^{-\delta_j})^{v_\infty(N_{E/K}(a))}) \\
& \times (\prod_{j=1}^n\langle \sigma_j(\gamma_{\psi_j}\pi^{-n(\eta_j)}) \rangle _v^{-z_j})^{v_\infty(N_{E/K}(a))}\,
\omega_v(\frac{N_{E/K}(a)}{\sgn (N_{E/K}(a))})^{-\delta_1n(\eta_1)-\cdots -\delta_nn(\eta_n)} \\
& \times \prod^s_{i=1} \rho_i(\frac{N_{E/K}(a)}{\sgn(N_{E/K}(a))})^{n(\psi_i)} \prod_{j=1}^n\langle \sigma_j(\frac{N_{E/K}(a)}{\sgn (N_{E/K}(a))}) \rangle _v^{-z_j n(\eta_j)}.
\end{align*}
%STep 2
Note that  there exist  integers $D,D'\in \mathbb R,$ such that, for $m\geq 0$, $j\in \{1, \ldots, r\}$ and $aO_E\in S_{m, \mathcal S_v}(\frak I_j, \mathcal B, N')$, we have: $$v_F(\prod_{i=1}^s\rho_i(N_{E/K}(a))^{n(\psi_i)})\geq -Dm-D'.$$ \par
%Step 3
By  the proof of Lemma \ref{LemmaS3.2}, there exists an integer $C\geq 1$ such that, for $m_1, \ldots m_n\in \mathbb N,$ with $m_1+\ldots +m_n \geq 1$, $m\geq \frac{C(\sum_{i=1}^n \ell_p(m_i))}{p-1}$ and $\delta\in \mathbb Z,$ we have:
\begin{align*}
\sum_{aO_E\in S_{m, \mathcal S_v}(\frak I_j, \mathcal B,N')}\omega_v(\frac{N_{E/K}(a)}{\sgn(N_{E/K}(a))})^{-\delta} \prod^s_{i=1} \rho_i(\frac{N_{E/K}(a)}{\sgn(N_{E/K}(a))})^{n(\psi_i)} \times \\
 \times \prod^n_{j=1} \sigma_j \Big(\frac{N_{E/K}(a)}{\sgn (N_{E/K}(a)) \omega_v(\frac{N_{E/K}(a)}{\sgn  (N_{E/K}(a))})} \Big)^{m_j}=0.
\end{align*}

%step 4
Now, let  $l\in \mathbb N,$  and select $k_1, \ldots, k_n \in \mathbb N\setminus\{0\}$ such that:\par
\noindent i) $\ell_p(k_i)\leq (l+1)(p-1),$\par
\noindent ii) $z_i\eta(\eta_i)+k_i\equiv 0\pmod{ p^{l+1}\mathbb Z_p}.$\par
\noindent  Then, there exists $D''\in \mathbb R,$ such that if $m\geq Cn(l+1) ,$ we have:
$$v_F(V_{m,j})\geq p^{l+1}{\rm Inf}\{ v_F(y_i), i=1, \ldots, n\}-D''m-D',$$
where $y_i=\sigma_i(\pi_v).$  We conclude as in the proof of Theorem \ref{TheoremS4.1}.
\end{proof}

%%%%%%%%%%%%%%%%Examples%%%%%%%%%%%%%%%%%%%%%%%%%%%%%%%%%%%%%%%%%%%%%%%%%%%%%%%%%%%%%

\section{Examples}\label{Examples}${}$\par

\subsection{The case $A=\mathbb F_q[\theta]$}\label{Simple}${}$\par

We take $\pi=\theta^{-1}.$ In that case  $d_\infty= 1,$ ${\rm Pic}(A)=\{ 1\},$ and we set  $$A_+= \{a\in A, a\, {\rm monic}\}=\{ a\in A, \sgn(a)=1\}.$$
Let $\varphi:\mathbb F_q\rightarrow \mathbb F_q, x\mapsto x^p.$ Let $a=\sum_{i=0}^m \lambda_i \theta^i , \lambda_i \in \mathbb F_q,$ we set:
$$\forall j\geq 0, \quad \varphi^j(a)=\sum_{i=0}^m \lambda_i ^{p^j}\theta ^i.$$
Let $s\geq 0,$ $n\geq 1,$ and $e_1, \ldots, e_s, l_1, \ldots , l_n \in \mathbb N.$ Let $(F,v_F)$ be a complete field with respect to a non-trivial valuation $v_F$ and such that $F$ is an $\mathbb F_q$-algebra. Recall that $\mathbb S_{F,n}=F^\times \times \mathbb Z_p^n.$ Let $E/K$ be a finite extension. Let $E_1/E$ be a finite abelian extension and let $(., E_1/E)$ be the Artin symbol. Let $\chi: {\rm Gal}(E_1/E) \rightarrow  F^\times$ be a group homomorphism. If $\frak P$ is a maximal ideal of $O_E,$ we set:
\begin{align*}
\chi (\frak P)= \begin{cases} 0 & \mbox{if $\frak P$ is ramified in $E_1/E$,} \\ \chi ((\frak P, E_1/E)) & \mbox{otherwise.} \end{cases}
\end{align*}
If we apply Theorem \ref{TheoremS4.1}, we get:
\begin{corollary}\label{CorollaryS5.1} Let $x_1, \ldots, x_s \in F, $ and let $y_1, \ldots, y_n\in F$ such that $v_F(y_i)<0, i=1, \ldots, n.$ Then, for $z_1, \ldots, z_n \in \mathbb Z_p$ and for $x\in F,$ the following  sum converges in $F:$
\begin{align*}
\sum_{m\geq 0}(\sum_{\substack{\frak I\in \mathcal I(O_E), \frak I\subset O_E, \\ N_{E/K}(\frak I) = aA, \\ a\in A_+, \deg_\theta a=m}} \chi (\frak I) \prod_{i=1}^s \varphi^{e_i} (a)\mid_{\theta= x_i} \prod_{j=1}^n \langle \varphi^{l_j}(a)\mid_{\theta=y_j} \rangle ^{z_j}) x^m,
\end{align*}
where for $y\in F, v_\infty (y)<0,$ and for $a\in A_+,$ $\langle a(y) \rangle = \frac{a(y)}{y^{\deg_\theta a}}.$
\end{corollary}\par
Let $X$ is an indeterminate over $K,$ and write:
$$\forall a\in A, \quad a(\theta+X)= \sum_{m\geq 0} a^{(m)} X^m, a^{(m)}\in A.$$
Then, for all $m\geq 0,$ $.^{(m)}:A\rightarrow A$ is an $\mathbb F_q$-linear map and we have:
$$\forall k\geq 0, \quad (\theta^k)^{(m)} ={k \choose m} \theta^{k-m},$$
where
\begin{align*}
{k \choose m}= \begin{cases} 0 & \mbox{if $k<m$,} \\ \frac{k!}{m!(k-m)!} \pmod{p} & \mbox{if $k\geq m$.} \end{cases}
\end{align*}
\noindent Observe that:
$$\forall i\geq 0, \forall m\geq 0, \forall a\in A, \quad \varphi^i(a^{(m)})= (\varphi^i(a))^{(m)}.$$

\begin{proposition}\label{PropositionS5.1}
Let $x_1, \ldots, x_s \in F, $ $m_1, \ldots, m_s\in \mathbb N,$ and let $y_1, \ldots, y_n\in F$ such that $v_F(y_i)<0, i=1, \ldots, n.$ Then, for $z_1, \ldots, z_n \in \mathbb Z_p$ and for $x\in F,$ the following  sum converges in $F:$
\begin{align*}
S(m_1, \ldots, m_s):=\sum_{m\geq 0}(\sum_{\substack{\frak I\in \mathcal I(O_E), \frak I\subset O_E, \\ N_{E/K}(\frak I) = aA, \\ a\in A_+, \deg_\theta a=m}} \chi (\frak I) \prod_{i=1}^s \varphi^{e_i} (a^{(m_i)})|_{\theta= x_i} \prod_{j=1}^n \langle  \varphi^{l_j}(a)|_{\theta=y_j}  \rangle ^{z_j}) x^m.
\end{align*}

Furthermore:
\begin{align*}
\lim_{m_1+\ldots+m_s\rightarrow +\infty} S(m_1, \ldots, m_s)=0.
\end{align*}
\end{proposition}

\begin{proof} Let $t_1, \ldots, t_s$ be $s$ indeterminates over $F$ and let $\mathbb T_s(F)$ be the Tate algebra in the variables $t_1, \ldots, t_s$ with coefficients in $F.$ Let $F'$ be the completion of the field of fraction of $\mathbb T_s(F).$ Let $x_1, \ldots, x_s \in F, $  $y_1, \ldots, y_n\in F$ such that $v_F(y_i)<0, i=1, \ldots, n,$ $ z_1, \ldots, z_n \in \mathbb Z_p, $ and  $ x\in F.$ By Corollary \ref{CorollaryS5.1}, the following sum converges in $F':$
\begin{align*}
S:=\sum_{m\geq 0}(\sum_{\substack{\frak I\in \mathcal I(O_E), \frak I\subset O_E, \\ N_{E/K}(\frak I) = aA, \\ a\in A_+, \deg_\theta a=m}} \chi (\frak I) \prod_{i=1}^s \varphi^{e_i} (a)\mid_{\theta= t_i+x_i} \prod_{j=1}^n \langle \varphi^{l_j}(a)\mid_{\theta=y_j} \rangle ^{z_j}) x^m.
\end{align*}

\noindent Since $\mathbb T_s(F)$ is closed in $F',$ we get $S\in \mathbb T_s(F).$ For $i=1, \ldots, s,$ we have:
$$\varphi^{e_i}(a)\mid_{\theta= t_i+x_i}=\sum_{m\geq 0}\varphi^{e_i}(a^{(m)})\mid_{\theta=x_i} t_i^m.$$
But, we have:
$$S=\sum_{m_1, \ldots, m_s\in \mathbb N} S(m_1, \ldots, m_s) t_1^{m_1}\cdots t_s^{m_s},$$
where $S(m_1, \ldots, m_s)\in F,$ and $\lim_{m_1+\ldots+m_s\rightarrow +\infty} S(m_1, \ldots, m_s)=0.$ Therefore, we get:
\begin{align*}
S(m_1, \ldots, m_s)=\sum_{m\geq 0}(\sum_{\substack{\frak I\in \mathcal I(O_E), \frak I\subset O_E, \\ N_{E/K}(\frak I) = aA, \\ a\in A_+, \deg_\theta a=m}} \chi (\frak I) \prod_{i=1}^s \varphi^{e_i} (a^{(m_i)})|_{\theta= x_i} \prod_{j=1}^n \langle  \varphi^{l_j}(a)|_{\theta=y_j}  \rangle ^{z_j}) x^m.
\end{align*}
\end{proof}
We refer the interested reader to \cite{AP} and \cite{ANT} for the arithmetic properties of a special case of the above  sums.\par
${}$\par
Recall that  $\mathbb C_\infty$ be the completion of $\overline{K}_\infty.$ Let $t_1, \ldots, t_s, z$ be $s+1$ indeterminates over $\mathbb C_\infty,$ and let $\mathbb T $ be the Tate algebra in the variables $t_1, \ldots, t_s,z$ with coefficients in $\mathbb C_\infty.$ Let $F$ be the completion of the field of fractions of $\mathbb T .$  Take in Corollary \ref{CorollaryS5.1}  $E_1=E=K,$  $e_1=\ldots= e_s=l_1=\ldots= l_n=0,$ $n=1,$ $x_i=t_i,$ $x=z,$ $y_1=\theta,$ $z_1=y\in \mathbb Z_p,$  we get that the following infinite sum  converges in $\mathbb T$ and is in fact an entire function on $\mathbb C_\infty ^{s+1}:$
$$\sum_{m\geq 0}(\sum_{a\in A_+, \deg_\theta a=m} \frac{a(t_1)\cdots a(t_s)}{\langle a \rangle ^y})z^m .$$
The above sums were introduced in \cite{PEL}. In particular, for all $n\in \mathbb Z,$ the sum $$L(n; \underline{t}; z):=\sum_{m\geq 0}(\sum_{a\in A_+, \deg_\theta a=m} \frac{a(t_1)\cdots a(t_s)}{a^n})z^m \in \mathbb T$$ defines an entire function on $\mathbb C_\infty^{s+1}$. Observe that, by Corollary \ref{CorollaryS3.1} and Corollary \ref{CorollaryS3.2}, if $n\leq 0,$ this sum is finite and furthermore it vanishes at $z=1$ if $s-n\equiv 0\pmod{q-1},$ $s-n\geq 1$. Using a special case of Anderson's log-algebraicity Theorem for $\mathbb F_q[\theta]$ (\cite{AND}, \cite{AND2}, \cite{THA} paragraphs 8.9 and 8.10, see also \cite{APTR}, \cite{APT}, \cite{ATR} and the forthcoming work of M. Papanikolas \cite{PAP}), F. Pellarin proved (\cite{PEL}), for $s=1,$  a formula connecting $L(1; t_1; 1)$ to a special function introduced by G. Anderson and D. Thakur (\cite{AND&THA}). This formula reflects an analytic class number formula \`a  la Taelman (\cite{TAE1}, \cite{TAE2}, \cite{FAN1}, \cite{FAN2}) for $L(1; \underline{t}; 1)$ (see \cite{APT}). Such an analytic class number formula has been generalized in \cite{DEM} to a larger class of $L$-series (in particular for $L(n; \underline{t}; z)$, for $n\geq 1$). We also refer the reader to \cite{AP}, \cite{AP2},  \cite{GOS3}, \cite{PEPER}, \cite{PEPER2}, \cite{PER1}, \cite{PER2}, \cite{PER3}, \cite{TAE3} for various arithmetic and analytic properties of the series $L(n; \underline{t}; z), n\in \mathbb Z.$
${}$\par
${}$\par

%%%%%%%%%%%%%

Now, let $P$ be a monic irreducible polynomial in $A$ of degree $d.$ Let $\mathbb C_P$ be the completion of an algebraic closure $K_P$ of the $P$-adic completion of $K$. Let $A_P$ be the valuation ring of $K_P.$ Then:
$$\forall a\in A_P^\times, \quad a=\omega_P(a) \langle a \rangle _P,  \text{ with } v_P(\langle a \rangle _P-1)\geq 1, \omega_P(a) \in \mathbb F_{q^d}^\times.$$
Note that $\frac{1}{q^d-1} \in \mathbb Z_p^\times,$ thus:
$$\forall a \in A_P^\times, \quad \langle a \rangle _P= (a^{q^d-1})^{\frac{1}{q^d-1}}.$$
Let $(F,v_F)$ be a complete field with respect to a non-trivial valuation $v_F$ and such that $F$ is an $\mathbb F_{q^d}$-algebra. Let $\sigma: K_P\hookrightarrow F$ be a continuous $\mathbb F_p$-algebra homomorphism. Then:
 $$\forall a\in A_+, \quad \sigma(\langle a \rangle _P)=(\sigma(a)^{q^d-1})^{\frac {1}{q^d-1}}.$$
 Let $y=\sigma (\theta),$ then $v_F(y)\geq 0,$   and $y\not \in \mathbb F_{q^d}.$ Furthermore, there exists $i\geq 0,$ such that:
 $$\forall a\in A, \quad \sigma (a)=\varphi ^i(a)\mid_{\theta=y}.$$
 Thus:
 $$v_F(\varphi^i(P)\mid_{\theta=y})>0.$$
 Furthermore:
 $$\sigma (\langle a \rangle _P)= ((\varphi^i(a)\mid_{\theta=y})^{q^d-1})^{\frac{1}{q^d-1}}=: \langle \varphi^i(a)\mid_{\theta=y} \rangle _P.$$
 Let $E/K$ be a finite extension and let $E_1/E$ be a finite abelian extension. Let $\chi: {\rm Gal}(E_1/E) \rightarrow F^\times$ be a group homomorphism. Let $\mathcal S_P$ be the set of maximal ideals of $O_E$ above $P.$ If we apply Theorem \ref{TheoremS4.2}, then, by the proof of Proposition \ref{PropositionS5.1}, we get:
\begin{corollary}\label{CorollaryS5.2}
Let $e_1, \ldots, e_s,m_1, \ldots, m_s, l_1, \ldots, l_n \in \mathbb N.$ Let $x_1, \ldots, x_s\in F.$ Let $y_1, \ldots, y_n \in F\setminus \mathbb F_{q^d}$ such that $v_F(\varphi^{l_i}(P)\mid_{\theta= y_i})>0.$ Then for $z_1, \ldots, z_n \in \mathbb Z_p, $ for $\delta\in \mathbb Z,$ and for  $x\in F,$ the following sum converges in $F:$
\begin{align*}
\sum_{m\geq 0}(\sum_{\substack{\frak I\in \mathcal I_{\mathcal S_P}(O_E), \frak I\subset O_E, \\ N_{E/K}(\frak I) = aA, \\ a\in A_+, \deg_\theta a=m}} \omega_P(a)^{\delta} \chi (\frak I) \prod_{i=1}^s \varphi^{e_i} (a^{(m_i)})|_{\theta= x_i} \prod_{j=1}^n \langle  \varphi^{l_j}(a)|_{\theta=y_j}  \rangle _P^{z_j}) x^m.
\end{align*}
\end{corollary}
${}$\par
Let $t_1, \ldots, t_s, z$ be $s+1$ indeterminates over $\mathbb C_P,$ and let $\mathbb T_P $ be the Tate algebra in the variables $t_1, \ldots, t_s,z$ with coefficients in $\mathbb C_P.$ Let $F_P$ be the completion of the field of fractions of $\mathbb T_P .$ Take in the above Corollary  $E_1=E=K,$  $e_1=\ldots= e_s=m_1=\ldots=m_s=l_1=\ldots= l_n=0,$ $n=1,$ $x_i=t_i,$ $x=z,$ $y_1=\theta,$ $z_1=y\in \mathbb Z_p,$ $\delta\in \mathbb Z,$ we get that the following sum  converges in $\mathbb T_P$ and is in fact an entire function on $\mathbb C_P^{s+1}:$
$$\sum_{m\geq 0}(\sum_{a\in A_+, \deg_\theta a=m, a\not \equiv 0\pmod{P}}\omega_P(a)^{\delta} \frac{a(t_1)\cdots a(t_s)}{\langle a \rangle _P^y})z^m .$$
 In particular, $\forall n\in \mathbb Z,$ the sum $$L_P(n; \underline{t}; z):=\sum_{m\geq 0}(\sum_{a\in A_+, \deg_\theta a=m, a\not \equiv 0\pmod{P}} \frac{a(t_1)\cdots a(t_s)}{a^n})z^m \in \mathbb T_P$$ defines an entire function on $\mathbb C_P^{s+1}.$ We  refer the reader to \cite{AT}, \cite{ATR}, for various arithmetic  properties of  ``special values" of the series $L_P(n; \underline{t}; z), n\in \mathbb Z.$
%%%%%%%%%%%%%%%%%%%%

\subsection{Twisted  Goss zeta functions}\label{GOSS}${}$\par

In this example, $A$ is ``general". Recall that  $\mathbb C_\infty$ be the completion of $\overline{K}_\infty.$  Let $s\geq 1$ be an integer. Let $\mathbb F\subset \overline{\mathbb F}_q$ be a  field containing $\mathbb F_\infty.$ Let $k_s(\mathbb F)$ be  defined as follows:\par
\noindent - if $s=1,$ $k(\mathbb F)=\mathbb F ((y_1))$ where $y_1$ is an indeterminate over $\mathbb F,$\par
\noindent - if $s\geq 2, $ let $y_s$ be an indeterminate over $k_{s-1}(\mathbb F),$ and set $k_s(\mathbb F)=k_{s-1}(\mathbb F)((y_s)).$\par
\noindent Observe that $\mathbb F$ is algebraically closed in $k_s(\mathbb F).$\par
  Let $L/K_\infty$ be a finite extension and let $O_L$ be the valuation ring of $L.$ Then:
$$O_L=\mathbb F_L[[\pi_L]],$$
where $\mathbb F_L$ is the residue field of $L,$ and $\mathbb \pi_L$ is a prime of $O_L.$ Let's consider the following tensor product:
$$k_s(\mathbb F_\infty)\otimes_{\mathbb F_\infty}O_L.$$
This ring can be identified naturally with $ k_s(\mathbb F_L)\otimes_{\mathbb F_L}O_L.$ Any element $f\in  k_s(\mathbb F_L)\otimes_{\mathbb F_L}O_L$ can be written in a unique way:
$$\sum_{i\geq 0}\alpha_i \otimes\pi_L ^i, \quad \alpha_i \in k_s(\mathbb F_L).$$
We set:
$$v_\infty(f)={\rm Inf} \{ v_\infty(\pi_L)i, i \in \mathbb N,\alpha_i\not =0\}.$$
Then $v_\infty$ is a valuation on $k_s(\mathbb F_\infty)\otimes_{\mathbb F_\infty}O_L$ which does not depend on the choice of $\pi_L.$ We denote by $k_s(\mathbb F_\infty)\widehat{\otimes}_{\mathbb F_\infty}O_L$ the completion of $k_s(\mathbb F_\infty)\otimes_{\mathbb F_\infty}O_L$ for $v_\infty.$ Finally denote by $L_s$  the completion (for $v_\infty$) of the field of fractions of $k_s(\mathbb F_\infty)\widehat{\otimes}_{\mathbb F_\infty}O_L.$ If we identify $1\otimes \mathbb F_L$ with $\mathbb F_L,$ and  $1\otimes \pi_L$ with $\pi_L$ which is thus an indeterminate over $ k_s(\mathbb F_L),$ we have:
$$L_s= k_s(\mathbb F_L)((\pi_L)).$$
If $L\subset L',$ then we have a natural injective map compatible with $v_\infty$:
$$L_s\hookrightarrow L'_s.$$
We denote by $\mathbb C_{\infty,s}$ the completion (for $v_\infty$) of the inductive limit $\varinjlim_{L/K_\infty \, {\rm finite}}L_s.$ Note that $\mathbb C_\infty, k_s(\overline{\mathbb F_q})\subset \mathbb C_{\infty,s},$ and the residue field of $\mathbb C_{\infty,s}$ is $k_s(\overline{\mathbb F_q}).$\par
Let $z$ be an indeterminate over $\mathbb C_{\infty,s},$ we denote by $\mathbb T_z(\mathbb C_{\infty,s})$ the Tate algebra in the variable $z$ with coefficients in $\mathbb C_{\infty,s}.$\par
${}$\par
We fix  a $K$-embedding of $\overline{K}$ in $\mathbb C_\infty.$ As in  section \ref{Notation}, we consider $[.]$ the Goss admissible map,   i.e. for  $I\in \mathcal I(A),$ we have:
$$[I]= \langle I \rangle \pi^{\frac{-\deg I}{d_\infty}}\in \pi^{\mathbb Q}U_\infty,$$
where $\pi$ is our fixed prime of $K_\infty.$ Note that  $V=K([I], I\in \mathcal I(A))$ is a finite extension of $K.$ Recall that $V$ is viewed as a subfield of $K_\infty ([I], I\in \mathcal I(A)).$ Let $O_V$ be the integral closure of $A$ in $V.$  Since, for $a\in K^\times,$ we have $[aA]=a/\sgn(a)$, we deduce that if $I$ is a non-zero ideal of $A$ then $[I]\in O_V.$ Also recall that $\mathbb F_\infty\subset O_V$ is algebraically closed in $V.$\par
Let $s\in \mathbb N.$ For $i=1, \ldots, s,$ let $\rho_i: K_\infty ([I], I\in \mathcal I(A))\rightarrow \mathbb F_\infty((y_i))$ be a continuous $\mathbb F_p$-algebra homomorphism. Let $\mathbb T_{\underline{\rho},z}(\mathbb C_\infty)$ be the closure of $\mathbb C_\infty[z][\rho_i(O_V), i=1, \ldots, s]$ in $\mathbb T_z(\mathbb C_{\infty,s}).$ Let also $\mathbb T_{\underline{\rho}}(\mathbb C_\infty)$ be the closure of $\mathbb C_\infty[\rho_i(O_V), i=1, \ldots, s]$ in $\mathbb C_{\infty,s}.$ \par
Let's observe that $\mathbb T_{\underline{\rho},z}(\mathbb C_\infty)$ is an affinoid algebra over $\mathbb C_\infty$ (this is of also the case for $\mathbb T_{\underline{\rho}}(\mathbb C_\infty)$).  Indeed, select $\theta\in A\setminus \mathbb F_q.$ Then there exist $v_1, \ldots, v_r \in  O_V$ such that:
$$O_V=\oplus_{j=1}^r \mathbb F_\infty[\theta] v_j.$$
For $i=1, \ldots, s,$ let $t_i=\rho_i(\theta).$ Then $t_1, \ldots, t_s,z$ are $s+1$ indeterminates over $\mathbb C_\infty.$ Let $\mathbb T_{s,z}(\mathbb C_\infty)$ be the Tate algebra in the variables $t_1, \ldots, t_s, z$ with coefficients in $\mathbb C_\infty.$ Clearly:
$$\mathbb T_{s,z}(\mathbb C_\infty) \subset \mathbb T_z(\mathbb C_{\infty,s}),$$
$$\mathbb T_{\underline{\rho},z}(\mathbb C_\infty)= \sum_{1\leq i_1, \ldots, i_s\leq r}\mathbb T_{s,z}(\mathbb C_\infty)\rho_1(v_{i_1})\cdots \rho_s(v_{i_s}).$$
If we apply Theorem \ref{TheoremS4.1}, we get as a special case:
\begin{corollary}\label{CorollaryS5.3} Let $E/K$ be a finite extension. Let $x\in \mathbb Z_p,$ the following sum converges in $\mathbb T_{\underline{\rho},z}(\mathbb C_\infty):$\par
$$\sum_{m\geq 0}( \sum_{\substack{\frak I\in \mathcal I(\mathcal O_E), \frak I\subset O_E \\ \deg N_{E/K} (\frak I)=m}} \rho_1([N_{E/K}(\frak I)]) \cdots \rho_s([N_{E/K}(\frak I)]) \langle N_{E/K}(\frak I) \rangle ^xz^m.$$
Furthermore, as a function in $z,$ it defines an entire function on $\mathbb C_\infty$ with values in $\mathbb T_{\underline{\rho}}(\mathbb C_\infty).$
\end{corollary}

%%%%%%%%%%%%%%%%%%%%%%%%%%%%%%%%%%%%%%%

\subsection{$A$-harmonic series attached to some admissible maps}\label{ANDERSON}${}$\par
 In this example, we work in the case $A$ general. Let $\eta : \mathcal I(A) \rightarrow \overline{K}_\infty^\times$ be an admissible map such that $n(\eta)\in \mathbb Z. $ Recall that there exist an open subgroup of finite index  $M(\eta)\subset K_\infty^\times,$ and an element $\alpha(\eta)\in \overline{K}_\infty^\times,$ such that:

$$\forall I\in \mathcal I(A), \forall a \in K^\times \cap M(\eta), \quad \eta (Ia)=\eta (I) (\frac{a}{\sgn(a)})^{n(\eta)} \alpha(\eta)^{v_\infty(a)}.$$
Note that there exists an open subgroup of finite index $N\subset K_\infty^\times,$ $N\subset M(\eta),$ such that:
$$\forall a \in K^\times \cap N, \quad \sgn (\alpha(\eta))^{v_\infty(a)}=1.$$
Let $\chi :\mathcal I(A) \rightarrow  \overline{K}_\infty^\times $ be the map defined by:
$$\forall I\in \mathcal I(A), \quad \chi (I)=\pi ^{v_\infty(\eta (I))+ \frac{\deg I}{d_\infty}(n(\eta)+v_\infty(\alpha(\eta)))}\sgn (\eta(I)).$$
If we set:
$$\mathcal P=\{ xA, x\in K^\times \cap N\}.$$ Then $\frac{\mathcal I(A)}{\mathcal P}$ is a finite abelian group and:
$$\forall I\in \mathcal I(A), \forall J\in \mathcal P, \quad \chi (IJ)= \chi (I).$$
Let $n\in \mathbb Z.$ For $u=(x, y)\in \mathbb C_\infty^\times \times \mathbb Z_p=:\mathbb S_\infty,$ let's set:
$$L_{\eta, A}(\chi^n; u)=\sum_{m\geq 0}\sum_{\substack{I\in \mathcal I(A), I\subset A \\ \deg I=m}} \chi^n (I)\langle \eta(I) \rangle ^{-y} x^{-m}.$$
Then, by Theorem \ref{TheoremS4.1}, $L_{\eta, A} (\chi^n; .)$ converges on $\mathbb S_\infty.$ Now, let $x\in \mathbb C_\infty^\times,$ and observe that:
$$L_{\eta, A} (\chi^{-n}; (\pi^{\frac{-n}{d_\infty}(n(\eta)+v_\infty(\alpha(\eta))} x,n))=\sum_{m\geq 0}\sum_{\substack{I\in \mathcal I(A), I\subset A \\ \deg I=m}} \eta(I)^{-n} x^{-m}.$$
Thus, for all $n\in \mathbb Z,$ the following function in the variable $z$ is entire on $\mathbb C_\infty$:
$$Z_{\eta,A}(n;z)=\sum_{m\geq 0}\sum_{\substack{I\in \mathcal I(A), I\subset A \\ \deg I=m}} \eta(I)^{-n} z^m .$$
Observe that, by Lemma \ref{LemmaS3.2}, if $m\, n(\eta)\leq 0,$ we have:
$$Z_{\eta,A}(m;z)\in K(\eta(I), I\in \mathcal I(A))[z].$$
Let's furthermore assume that $\eta$ is an algebraic admissible map. Let $E=K(\eta(I), I\in \mathcal I(A))$ which is a finite extension of $K,$ and let $E_\infty =E\otimes_KK_\infty.$ Let $z$ be an indeterminate over $K_\infty,$ following \cite{ATR}, we denote by $\mathbb T_z(E_\infty)$ the closure of $E_\infty[z]$ in $E\otimes_K \widetilde{K_\infty},$ where $\widetilde{K_\infty}$ denotes the completion of the field of fractions of the Tate algebra in the variable $z$ with coefficients in $K_\infty.$ From the above discussion, we get:
\begin{corollary}\label{CorollaryS5.4} Let $m\in \mathbb Z.$ The following sum converges in $\mathbb T_z(E_\infty):$
$$\zeta_{\eta, A}(m; z):=\sum_{k\geq 0}\sum_{\substack{I\in \mathcal I(A), I\subset A \\ \deg I=m}} \eta(I)^{-m}  z^k .$$
Furthermore, if $m\,  n(\eta)\leq 0,$ we have:
$$\zeta_{\eta, A}(m; z)\in E[z].$$
%In particular, we have:
%$$\zeta_{\eta, A}(1; 1)\in E_\infty.$$
\end{corollary}\par
${}$\par
Let $v$ be a finite place of $K,$ and let $P_v$ be the maximal ideal of $A$ associated to $v.$ Let $\mathcal I(P_v)$ be the group of fractional ideals of $A$ which are relatively prime to $P_v.$ Let $\mathbb C_v$ be the $v$-adic completion of $\overline{K}_v.$ Let $\eta$ be our  admissible map introduced in the beginning of the paragraph, and we assume that $\eta$ is algebraic.  Let $\chi: \mathcal I(P_v)\rightarrow \overline{K}_v$ such that:
$$\forall I\in \mathcal I(P_v), \quad \chi (I) = \omega_v(\eta(I))\pi_v^{v(\eta (I))+\frac{\deg I}{d_\infty} v(\alpha (\eta))}.$$
Note that there exists an open subgroup of finite index $N'\subset K_\infty^\times,$ $N'\subset M(\eta),$ such that:
$$\forall a \in K^\times \cap N', \quad \omega_v(\alpha(\eta))^{v_\infty(a)}=1.$$
Then:
$$\forall I\in I(P_v), \forall J\in \mathcal P', \quad \chi (IJ)=\chi(I),$$
where $$\mathcal P'=\{xA\in \mathcal I(P_v), x\equiv 1\pmod{P_v}, x\in K^\times \cap N'\cap {\rm Ker} \, \sgn\}.$$
Let $n\in \mathbb Z.$ For $u=(x, y)\in \mathbb C_v^\times \times \mathbb Z_p,$ let's set:
$$L_{v,\eta, A}(\chi^n; u)=\sum_{m\geq 0}\sum_{\substack{I\in \mathcal I( P_v), I\subset A \\ \deg I=m}} \chi^n (I)\langle \eta(I) \rangle _v^{-y} x^{-m}.$$
Then, by Theorem \ref{TheoremS4.2}, $L_{v,\eta, A} (\chi^n; .)$ converges on $\mathbb C_v^\times \times \mathbb Z_p.$ Now, let $x\in \mathbb C_v^\times,$ and observe that:
$$L_{v,\eta, A} (\chi^{-n}; ( \pi_v^{\frac{-n}{d_\infty}v(\alpha(\eta))}x,n))=\sum_{m\geq 0}\sum_{\substack{I\in \mathcal I( P_v), I\subset A \\ \deg I=m}} \eta(I)^{-n} x^{-m}.$$
Thus, for all $n\in \mathbb Z,$ the following function in the variable $z$ is entire on $\mathbb C_v$:
$$Z_{v,\eta,A}(n;z)=\sum_{m\geq 0}\sum_{\substack{I\in \mathcal I( P_v), I\subset A \\ \deg I=m}} \eta(I)^{-n} z^m .$$
By Lemma \ref{LemmaS3.2}, if $m\, n(\eta)\leq 0,$ then:
$$Z_{v,\eta,A}(m;z)\in K(\eta(I), I\in \mathcal I(A))[z].$$
In particular, if $P_v=\alpha A,$ $\alpha \in M(\eta)$ (there exist infinitely many such maximal ideals by Chebotarev's density Theorem), then, for $m\in \mathbb Z$ we get:
$$Z_{v,\eta,A}(m;z)= \Big(1-(\frac{\alpha}{\sgn (\alpha)})^{-mn(\eta)}\alpha(\eta)^{m\frac{\deg P_v}{d_\infty}} z^{\deg P_v} \Big) Z_{\eta,A}(m;z).$$
We refer the interested reader to a forthcoming work of the authors dedicated to the arithmetic of such $A$-harmonic series (\cite{ANT2}).

%%%%%%%%%%%%%%%Appendix:Multiple several variable zeta functions%%%%%%%%%%%%%%%%%%%%%%%%%%%%%%%%%%%%%%%%%%%

\section{Multiple several variable twisted zeta functions}\label{MSVTZF}${}$\par

We briefly explain how the constructions in Section \ref{SVTZF} can easily be generalized in the spirit of Thakur's construction of positive characteristic multiple zeta values; the reader interested by the arithmetic of multiple zeta values for $K=\mathbb F_q(\theta)$ is referred to (this list is not exhaustive): \cite{AND&THA2}, \cite{THA2},\cite{THA3}, \cite{LATH}, \cite{LATH2}, \cite{CHA}, \cite{CPY}, \cite{PEL2}. We keep the notation of Section \ref{SVTZF}.\par
${}$\par
Let $r\in \mathbb N,r\geq 1.$ Let $s_1, \ldots, s_r\in \mathbb N$ and $n_1, \ldots, n_r\in \mathbb N\setminus\{0\}.$ Let $(F,v_F)$ be a field which is complete with respect to a non-trivial valuation $v_F: F\rightarrow \mathbb R\cup\{+\infty\}$ and such that $F$ is an $\overline{\mathbb F}_q$-algebra.  We set:
 $$\mathbb S_{F,(n_1,\ldots, n_r)}=(F^\times\times \mathbb Z_p^{n_1})\times \cdots \times (F^\times\times  \mathbb Z_p^{n_r}).$$\par
For $i=1, \ldots, r,$ let $\psi_{1, i}, \ldots, \psi_{s_i,i}$ be $s_i$ admissible maps such that $n(\psi_{j,i})\in \mathbb N.$ For $i=1, \ldots, r, j=1, \ldots,s_i,$ let $\rho_{j,i}:K(\psi_{j,i})\rightarrow F$ be an $\mathbb F_p$-algebra homomorphism. For $i=1, \ldots,r,$ we set:
$$\forall I\in \mathcal I(A), \quad \underline{\rho}_i(\underline{\psi}_i(I))=\prod_{j=1}^{s_i}\rho_{j,i}(\psi_{j,i}(I))\in F.$$\par
Let $E/K$ be a finite extension. Let $\mathcal B\subset O_E$ be a non-zero ideal. Let $v_1, \ldots, v_n$ be the places of $E$ above $\infty,$ and for $ j=1, \ldots, n,$ let $N_{j}$ be an open subgroup of finite index of $E_{v_j}^\times$ ($E_{v_j}$ is the $v_j$-adic completion of $E$), and we set : $N=\prod_{j=1, \ldots, n} N_{j}.$ Let $\mathcal S$ be a finite set, possibly empty,  of maximal ideal of $O_E$ which are relatively prime to $\mathcal B.$ For $i=1, \ldots, r,$ let $\chi_i: \mathcal I(\mathcal B) \rightarrow F$ be a map such that:
$$\forall \frak I\in\mathcal I_(\mathcal B), \forall \frak J\in  \mathcal P(\mathcal B, N), \quad \chi_i(\frak I
\frak J)=\chi_i(\frak I).$$\par
For $i=1, \ldots, r,$  let $\eta_{1, i}, \ldots, \eta_{n_i,i}$ be $n_i$ admissible maps  with values in $\overline{K}_\infty^\times.$ For $i=1, \ldots, r, j=1, \ldots,n_i,$ let $\sigma_{j,i}:K(\langle \eta_{j,i} \rangle )\rightarrow F$ be a continuous $\mathbb F_p$-algebra homomorphism. For $i=1, \ldots,r,$ let $u_i=(x_i; z_{1, i},\ldots, z_{n_i,i})\in F^\times \times \mathbb Z_p^{n_i},$ for $I\in \mathcal I(A)$, we set:
$$I_{\underline{\sigma}_i, \underline{\eta}_i}^{u_i}= x_i^{\deg I}\prod_{j=1}^{n_i}\sigma_{j,i}(\langle \eta_{j,i}(I) \rangle )^{z_{j,i}}\in F^\times.$$\par
We fix $i\in \{1, \ldots, r\}.$ Let $d\geq 0$ be an integer, for $u_i\in F^\times \times \mathbb Z_p^{n_i},$  we set:
$$S_{d; \mathcal S, O_E}(\underline{\rho}_i, \underline{\psi}_i; \underline{\sigma}_i, \underline{\eta}_i;\chi_i; u_i )=\sum_{\substack{\frak I \in \mathcal I_{\mathcal S}(\mathcal B), \frak I\subset O_E \\ \deg N_{E/K}(\frak I)=d}} \chi_i(\frak I) \underline{\rho}_i(\underline{\psi}_i(N_{E/K}(\frak I)))N_{E/K}(\frak I)_{\underline{\sigma}_i, \underline{\eta}_i}^{-u_i}.$$
By the proof of Theorem \ref{TheoremS4.1} and Remark \ref{RemarkS4.1}, we get:
\begin{corollary}\label{CorollaryS6.1} Let $u=(u_1, \ldots, u_r)\in \mathbb S_{F, (n_1, \ldots, n_r)},$ then the following sums converge in $F:$
$$\sum_{d\geq 0}S_{d; \mathcal S, O_E}(\underline{\rho}_1, \underline{\psi}_1; \underline{\sigma}_1, \underline{\eta}_1;\chi_1; u_1 )\sum_{d>d_2>\ldots >d_{r}\geq 0}\prod_{j=2}^rS_{d_j; \mathcal S, O_E}(\underline{\rho}_j, \underline{\psi}_j; \underline{\sigma}_j, \underline{\eta}_j;\chi_j; u_j ),$$
$$\sum_{d\geq 0}S_{d; \mathcal S, O_E}(\underline{\rho}_1, \underline{\psi}_1; \underline{\sigma}_1, \underline{\eta}_1;\chi_1; u_1 )\sum_{d\geq d_2\geq \ldots \geq d_{r}\geq 0}\prod_{j=2}^rS_{d_j; \mathcal S, O_E}(\underline{\rho}_j, \underline{\psi}_j; \underline{\sigma}_j, \underline{\eta}_j;\chi_j; u_j ).$$
\end{corollary}
${}$\par

We furthermore assume that  $n(\eta_{j,i})\in \mathbb Z , j=1, \ldots, n_i, i=1, \ldots, r,$ and that  all our admissible maps are algebraic. Let $v$ be a finite place of $K,$ and  Let $P_v$ be the maximal ideal of $A$ corresponding to the finite place $v.$\par
Let's set:
$$\mathbb S_{v,F,(n_1, \ldots, n_r)}= (F^\times \times \mathbb Z_p^{n_1} \times {\mathbb Z}^{n_1})\times \cdots \times (F^\times \times \mathbb Z_p^{n_r} \times {\mathbb Z}^{n_r}) .$$\par
For $i=1, \ldots, s, j=1, \ldots, s_i,$ let  $\rho_{j,i}: K(\psi_i)\hookrightarrow F$ be an $\mathbb F_p$-algebra homomorphism. For $i=1, \ldots,r,$ we set as above:
$$\forall I\in \mathcal I(A), \quad \underline{\rho}_i(\underline{\psi}_i(I))=\prod_{j=1}^{s_i}\rho_{j,i}(\psi_{j,i}(I))\in F.$$\par
For $i=1, \ldots, r, $ $j=1, \ldots, n_i,$  let $\sigma_{j,i}:  K_v(\langle \eta_i \rangle _v)\hookrightarrow F$ be a continuous $\mathbb F_p$-algebra homomorphism. For $i\in \{1, \ldots, r\},$ for $u_i=(x; z_1, \ldots, z_{n_i}; \underline{\delta}_i)\in F^\times \times \mathbb Z_p^{n_i}\times \mathbb Z^{n_i},$ for $I\in \mathcal I(A),$ we set:
$$I_{\underline{\sigma}_i, \underline{\eta}_i}^{u_i}= x^{\deg I} \prod_{j=1}^{n_i} \omega_v( \eta_{j,i}(I))^{\delta_{j,i}} \prod_{j=1}^{n_i} \sigma_{j,i}(\langle \eta_{j,i}(I) \rangle _v^{z_j})\in F^\times.$$\par
Let $E/K$ be a finite extension, and let $\mathcal B,N, \chi_1, \ldots, \chi_r,$ and $\mathcal S$ as above.  Let $\mathcal S_v$ be the union of $S$ and the maximal ideals of $O_E$ above $P_v$ and that do not divide $\mathcal B.$ \par
We fix $i\in \{1, \ldots, r\}.$ Let $d\geq 0$ be an integer, for $u_i\in F^\times \times \mathbb Z_p^{n_i}\times \mathbb Z^{n_i},$  we set:
$$S_{v;d; \mathcal S, O_E}(\underline{\rho}_i, \underline{\psi}_i; \underline{\sigma}_i, \underline{\eta}_i;\chi_i; u_i )=\sum_{\substack{\frak I \in \mathcal I_{\mathcal S}(\mathcal B), \frak I\subset O_E \\ \deg N_{E/K}(\frak I)=d}} \chi_i(\frak I) \underline{\rho}_i(\underline{\psi}_i(N_{E/K}(\frak I)))N_{E/K}(\frak I)_{\underline{\sigma}_i, \underline{\eta}_i}^{-u_i},$$

$$S_{v;d; \mathcal S_v, O_E}(\underline{\rho}_i, \underline{\psi}_i; \underline{\sigma}_i, \underline{\eta}_i;\chi_i; u_i )=\sum_{\substack{\frak I \in \mathcal I_{\mathcal S_v}(\mathcal B), \frak I\subset O_E \\ \deg N_{E/K}(\frak I)=d}} \chi_i(\frak I) \underline{\rho}_i(\underline{\psi}_i(N_{E/K}(\frak I)))N_{E/K}(\frak I)_{\underline{\sigma}_i, \underline{\eta}_i}^{-u_i}.$$

By the proof of Theorem \ref{TheoremS4.2}, we get:

\begin{corollary}\label{CorollaryS6.2} Let $u=(u_1, \ldots, u_r)\in \mathbb S_{v,F, (n_1, \ldots, n_r)},$ then the following sums converge in $F:$
$$\sum_{d\geq 0}S_{v; d; \mathcal S_v, O_E}(\underline{\rho}_1, \underline{\psi}_1; \underline{\sigma}_1, \underline{\eta}_1;\chi_1; u_1 )\sum_{d>d_2>\ldots >d_{r}\geq 0}\prod_{j=2}^rS_{v;d_j; \mathcal S, O_E}(\underline{\rho}_j, \underline{\psi}_j; \underline{\sigma}_j, \underline{\eta}_j;\chi_j; u_j ),$$
$$\sum_{d\geq 0}S_{v;d; \mathcal S_v, O_E}(\underline{\rho}_1, \underline{\psi}_1; \underline{\sigma}_1, \underline{\eta}_1;\chi_1; u_1 )\sum_{d\geq d_2\geq \ldots \geq d_{r}\geq 0}\prod_{j=2}^rS_{v;d_j; \mathcal S, O_E}(\underline{\rho}_j, \underline{\psi}_j; \underline{\sigma}_j, \underline{\eta}_j;\chi_j; u_j ).$$
\end{corollary}
${}$\par
Let's pursue  our basic example \ref{ANDERSON}. Let $\eta: \mathcal I(A)\rightarrow \overline{K}_\infty^\times$ be an admissible map such that $n(\eta)\in \mathbb Z.$ Let $r\geq 1$ be an integer and let $z_1, \ldots, z_r$ be $r$ indeterminates over $\mathbb C_\infty.$ Let $\underline{n}=(n_1, \ldots, n_r)\in \mathbb Z^r.$ Let's define  $Z_{\eta, A}(\underline{n}; \underline{z})\in K(\eta_i(I), I\in \mathcal I(A), i=1, \ldots, r)[z_2, \ldots, z_r][[z_1]]$ to be the following sum:
$$\sum_{d\geq0} \sum_{\substack{I_1, \ldots, I_r\in \mathcal I(A) \\ I_1, \ldots, I_r\subset A \\d=\deg I_1>\deg I_2>\cdots >\deg I_r\geq 0}} \frac{1}{\eta_1(I_1)^{n_1}\cdots \eta_r(I_r)^{n_r}} z_1^{\deg I_1}\cdots z_r^{\deg I_r};$$
let's also define $Z^*_{\eta, A}(\underline{n}; \underline{z})\in K(\eta_i(I), I\in \mathcal I(A), i=1, \ldots, r)[z_2, \ldots, z_r][[z_1]]$ as the following sum:
$$\sum_{d\geq0}\sum_{\substack{I_1, \ldots, I_r\in \mathcal I(A) \\ I_1, \ldots, I_r\subset A \\ d=\deg I_1\geq \deg I_2\geq \cdots \geq \deg I_r\geq 0}}\frac{1}{\eta_1(I_1)^{n_1}\cdots \eta_r(I_r)^{n_r}} z_1^{\deg I_1}\cdots z_r^{\deg I_r}.$$
Then, as in example \ref{ANDERSON}, by Corollary \ref{CorollaryS6.1}, we deduce that $Z_{\eta, A}(\underline{n}; \underline{z})$ and $Z^*_{\eta, A}(\underline{n}; \underline{z})$ define entire functions on $\mathbb C_\infty^r.$ Let's also observe that, by Lemma \ref{LemmaS3.2}, if $n_1n(\eta)\leq 0, $ then:
$$Z_{\eta, A}(\underline{n}; \underline{z}),Z^*_{\eta, A}(\underline{n}; \underline{z})\in K(\eta_i(I), I\in \mathcal I(A), i=1, \ldots, r)[z_1, z_2, \ldots, z_r].$$
Let's observe that, if $\eta=[.], $ $A=\mathbb F_q[\theta],$ and $\pi=\frac{1}{\theta},$ we recover Thakur's multiple zeta values (as  in example \ref{Simple}, one  can also recover their deformations in Tate algebras treated in   \cite{PEL2}):
\begin{align*}
Z_{\eta, A}(\underline{n}; \underline{z})=\sum_{d\geq0} \sum_{\substack{a_1, \ldots, a_r\in A_+ \\ d=\deg a_1>\cdots >\deg a_r\geq 0}} \frac{1}{a_1^{n_1}\cdots  a_r^{n_r}} \prod_{i=1}^rz_i^{\deg_\theta a_i}\in K[z_2, \ldots, z_r][[z_1]].
\end{align*}
Let's assume that $\eta $ is algebraic. Let $v$ be a finite place of $K,$ and let $P_v$ be the maximal ideal of $A$ associated to $v.$ Let $\mathcal I(P_v)$ be the group of fractional ideals of $A$ which are relatively prime to $P_v.$ Let $\mathbb C_v$ be the $v$-adic completion of $\overline{K}_v.$ We define  $Z_{v,\eta, A}(\underline{n}; \underline{z})\in K(\eta_i(I), I\in \mathcal I(A), i=1, \ldots, r)[z_2, \ldots, z_r][[z_1]]$ to be the following sum:
$$\sum_{d\geq0}\sum_{\substack{I_1\in \mathcal I(P_v), I_2, \ldots, I_r\in \mathcal I(A) \\ I_1, \ldots, I_r\subset A \\ d=\deg I_1>\cdots >\deg I_r\geq 0}}\frac{1}{\eta_1(I_1)^{n_1}\cdots \eta_r(I_r)^{n_r}} z_1^{\deg I_1}\cdots z_r^{\deg I_r};$$
we also define $Z^*_{v,\eta, A}(\underline{n}; \underline{z})\in K(\eta_i(I), I\in \mathcal I(A), i=1, \ldots, r)[z_2, \ldots, z_r][[z_1]]$ to be:
$$\sum_{d\geq0}\sum_{\substack{I_1\in \mathcal I(P_v), I_2, \ldots, I_r\in \mathcal I(A) \\ I_1, \ldots, I_r\subset A \\  d=\deg I_1\geq \cdots \geq \deg I_r\geq 0}} \frac{1}{\eta_1(I_1)^{n_1}\cdots \eta_r(I_r)^{n_r}} z_1^{\deg I_1}\cdots z_r^{\deg I_r}.$$
Again, as in example \ref{ANDERSON}, by Corollary \ref{CorollaryS6.2}, we deduce that $Z_{v,\eta, A}(\underline{n}; \underline{z})$ and $Z^*_{v,\eta, A}(\underline{n}; \underline{z})$ define entire functions on $\mathbb C_v^r.$\par
For example,  if $\eta=[.], $ $A=\mathbb F_q[\theta],$ and $\pi=\frac{1}{\theta},$ we have:
$$Z_{v,\eta, A}(\underline{n}; \underline{z})=\sum_{d\geq0} \sum_{\substack{a_1, \ldots, a_r\in A_+, a_1\not \in P_v \\ d=\deg a_1>\deg a_2>\cdots >\deg a_r\geq 0}}\frac{1}{a_1^{n_1}\cdots  a_r^{n_r}} \prod_{i=1}^rz_i^{\deg_\theta a_i}.$$
Note that, if $n_1, \ldots, n_r\leq 0,$ then:
$$Z_{v,\eta, A}(\underline{n}; \underline{z})\equiv Z_{\eta, A}(\underline{n}; \underline{z})\pmod{P^{n_1}A[z_1, \ldots, z_r]}.$$
Furthermore, for $\underline{n}\in \mathbb Z^r,$  we observe  that $Z_{v,\eta, A}(\underline{n}; \underline{z})$ is the $P_v$-adic limit of certain sequences $Z_{\eta, A}((m_k,n_2, \ldots, n_r); \underline{z})\in K[z_1, \ldots, z_r],$ $m_k\leq 0,$ where $m_k$  is suitably chosen and $m_k$ converges $p$-adically to $n_1.$ Indeed, let $\mathbb T_{\underline{z}}( K_v)$ be the Tate algebra in the variables $z_1, \ldots, z_r,$ with coefficients in $K_v$ equipped with the Gauss valuation associated to $v$ and still denoted by $v.$ Observe that, if $m\leq 0,$ we have:
$$\deg_{z_1}Z_{\eta, A}((m, n_2, \ldots, n_r); \underline{z})\leq \frac{\ell_q(-m)}{q-1},$$
where $\ell_q(-m)$ is the sum of the $q$-adic digits of $-m.$ Now, as in the proof of Theorem \ref{TheoremS4.2}, for $k\geq 0,$ let's select
$-m_k\in \mathbb N$ such that:

i) $-m_k\equiv -n_1\pmod{q^{k+1}\mathbb Z_p},$

ii)$-m_k\equiv -n_1\pmod{q^{\deg P_v}-1},$

iii) $\ell_q(-m_k)\leq (k+\deg P_v)(q-1),$

iv) $-m_k\geq q^{k+1}.$\par
\noindent For example, if $-n_1=\sum_{i\geq 0} a_i q^i,$ $a_i\in \{0, \ldots, q-1\},$ select $\delta_k\in \{1, \ldots, q^{\deg P_v}-1\}$ such that $-n_1-\sum_{i= 0}^k a_i q^i \equiv \delta_k \pmod{q^{\deg P_v}-1},$ and set:
$$-m_k=\sum_{i= 0}^k a_i q^i+\delta_k q^{(k+1)\deg P_v}.$$
\noindent Then:
\begin{align*}
&v(Z_{\eta, A}((m_k, n_2, \ldots, n_r);\underline{z})-\sum_{d=0}^{k+\deg P_v}\sum_{\substack{a_1, \ldots, a_r\in A_+, a_1\not \in P_v \\ d=\deg a_1>\cdots >\deg a_r\geq 0}} \frac{1}{a_1^{n_1}\cdots  a_r^{n_r}} \prod_{i=1}^rz_i^{\deg_\theta a_i}) \\
&\quad \geq q^{k+1}-(\mid n_2\mid +\cdots +\mid n_r\mid )(\deg P_v +k).
\end{align*}
\noindent Thus in $\mathbb T_{\underline{z}}(K_v):$
$$Z_{v,\eta, A}(\underline{n}; \underline{z})=\lim_kZ_{\eta, A}((m_k, n_2, \ldots, n_r);\underline{z}).$$

To our knowledge, and  in the case $A=\mathbb F_q[\theta],$  these type of elements   $Z_{v,\eta, A}(\underline{n}; \underline{1})\in K_v$ have not been studied so far, and it would be very interesting to obtain informations on such type  of objects, especially  in the spirit of \cite{AND&THA2}.

%%%%%%%%%%%%%%References%%%%%%%%%%%%%%%%%%%%%%%%%%%%%%%%%%%%%%%%%%%%%%%%%%%%%%%%%%%%%%

\end{document}